\documentclass{amsart}
\usepackage{cleveref}
\usepackage{amssymb}
\usepackage{fullpage}
\usepackage{xcolor}
\usepackage[colorinlistoftodos,bordercolor=orange,backgroundcolor=orange!20,linecolor=orange,textsize=scriptsize]{todonotes}

 \renewcommand{\ge}{\geqslant}
 \renewcommand{\le}{\leqslant}

  \newcommand{\y}{\mathbf{y}}
  \newcommand{\x}{\mathbf{x}}
  \newcommand{\z}{\mathbf{z}}
  
  \newcommand{\0}{\mathbf{0}}
 \newcommand\ba{{\mathbf a}}

 \newcommand\ff{{\mathbf f}}
  \newcommand\bF{{\mathbf F}}
 \newcommand\bg{{\mathbf g}}
  \newcommand\bG{{\mathbf G}}
 \newcommand\bi{{\mathbf i}}
 \newcommand\bj{{\mathbf j}}
 \newcommand\bk{{\mathbf k}}
 \newcommand\bl{{\mathbf l}}

 \newcommand\bH{{\mathbf H}}

 \newcommand\bu{\mathbf{u}}
 \newcommand\bU{\mathbf{U}}
 \newcommand\bv{\mathbf{v}}
 \newcommand\bV{{\mathbf V}}
 \newcommand\bw{{\mathbf w}}
 \newcommand\bx{{\mathbf x}}
 \newcommand\by{{\mathbf y}}

 \newcommand\cE{{\mathcal E}}

 \newcommand\cJ{{\mathcal J}}

 \newcommand\cM{{\mathcal M}}
 
 \newcommand\cO{{\mathcal O}}

 \newcommand\cS{{\mathcal S}}
 \newcommand\cT{{\mathcal T}}

 \newcommand\cW{{\mathcal W}}
 
 \newcommand\cY{{\mathcal Y}}

 \newcommand\rD{{\rm D}}
 
 \newcommand\rH{{\rm H}}

 \newcommand\rP{{\rm P}}
 
 \newcommand\rS{{\rm S}}

 \newcommand\rV{{\rm V}} 
 
 \newcommand{\e}{\mathbf{e}}
 \newcommand{\f}{\mathbf{f}}

 \newcommand{\m}{\mathbf{m}}
 \newcommand{\n}{\mathbf{n}}

  \newcommand{\trans}{^\top}

\newcommand{\diag}{\operatorname{diag}}

\newcommand{\C}{\mathbb{C}}
\newcommand{\F}{\mathbb{F}}
\newcommand{\bbP}{\mathbb{P}}
\newcommand{\R}{\mathbb{R}}
\newcommand{\N}{\mathbb{N}}
\newcommand{\Q}{\mathbb{Q}}

\newcommand{\Z}{\mathbb{Z}}

\newtheorem{theorem}{Theorem}
\newtheorem{proposition}[theorem]{Proposition}
\newtheorem{lemma}[theorem]{Lemma}
\newtheorem{corollary}[theorem]{Corollary}
\theoremstyle{definition}
\newtheorem{definition}[theorem]{Definition}

\theoremstyle{remark}

\newtheorem{example}[theorem]{Example}

\begin{document}
\title{Spectral norm of a symmetric tensor and its computation}
\author[S.~Friedland]{Shmuel~Friedland}
\address{Department of Mathematics, Statistics and Computer Science,  University of Illinois, 851 South Morgan Street, Chicago, Illinois 60607-7045, USA}
\email{friedlan@uic.edu}
\author[L.~Wang]{Li Wang}
\address{Department of Mathematics, University of Texas at Arlington, 411 S. Nedderman Drive, 478 Pickard Hall, Arlington, Texas 76019-0408, USA}
\email{li.wang@uta.edu}

  \date{January 11, 2020}

\maketitle
\begin{abstract}We show that the spectral norm of a $d$-mode real or complex symmetric tensor in $n$ variables can be computed by finding the fixed points of the corresponding polynomial map. For a generic complex symmetric tensor the number of fixed points is finite, and we give  upper and lower bounds for the number of fixed points.  For $n=2$ we show that these fixed points are the roots of a corresponding univariate polynomial of degree at most $(d-1)^2+1$, except certain cases, which are completely analyzed.  In particular, for $n=2$ the spectral norm of $d$-symmetric tensor is polynomially computable in $d$ with a given relative precision.  For a fixed $n>2$ we show that the spectral norm of a $d$-mode symmetric tensor is  polynomially computable in $d$ with a given relative precision with respect to the Hilbert-Schmidt norm of the tensor.  These results show that the geometric measure of entanglement of $d$-mode symmetric qunits on $\C^n$ are polynomially computable for a fixed $n$.
\end{abstract}

\section{Introduction}\label{sec:intro}
The spectral norm of a matrix, i.e., the maximal singular value, has numerous applications in pure and applied mathematics.
One of the fundamental reasons for the tremendous use of this norm is that it is polynomial-time computable and the software for its computation is easily available 
on MAPLE, MATHEMATICA, MATLAB and other platforms.

Multiarrays, or $d$-mode tensors, i.e. $d\ge 3$, are starting to gain popularity due to data explosion and other applications.  Usually, these problems deal with real valued tensors.
Since the creation of quantum physics, $d$-mode tensors over complex numbers became the basic tool in treating the $d$-partite states.  Furthermore,
the special case of $d$-partite symmetric qubits, called bosons, is the basic ingredient in construction the boson sampling devices \cite{AA13, Netall}.  

The ($\F$-)spectral norm of a tensor is a well defined quantity over the real ($\F=\R$) or complex numbers ($\F=\C$).  
Unlike in the matrix case, the computation of the spectral norm in general can be NP-hard \cite{FL18,HL13}.  Furthermore, the complex spectral norm of a real tensor can be bigger than its real spectral norm.   In spite of these numerical difficulties, there is a need to compute these norms in special cases of interesting applications, as the geometric measure of entanglement. (See later in the Introduction and \S\ref{sec: spwcnrm}.) Even the simplest case of $d$-partite qubits poses  theoretical and numerical challenges \cite{GFE09}.  This can be partly explained by the fact that 
the space $\otimes^d\C^2$ has dimension $2^d$.

In this paper we restrict ourselves to $d$-mode symmetric tensors over $\F^n$, denoted as $\rS^d\F^n$.  The dimension of this space is ${n+d- 1\choose d}={n+d-1\choose n-1}$.  Hence for a fixed $n$ this dimension is $O(d^{n-1})$.  In particular, the dimension of $\rS^d\C^2$ is $d+1$.
A symmetric tensor $\cS\in\rS^d\F^n$ can be identified with a homogeneous
polynomial $f=f_{\cS}$ of degree $d$ in $n$ variables over $\F$, denoted here as $\rP(d,n,\F)$.  It was already observed by J. J. Sylvester \cite{Syl51} that binary forms, i.e., $n=2$, posses very special properties related to polynomials of one complex variable.  
The spectral norm of $\cS\in\rS^d\F^n$ is denoted by $\|\cS\|_{\sigma,\F}$.  Its value is equal to the following maximum of $f\in\rP(d,n,\F)$ on the unit sphere in $\F^n$:
\[\|f\|_{\sigma,\F}=\max\{|f(\x)|,\;\x\in\F^n, \|\x\|=1\}.\]

The spectral norm of a complex valued symmetric tensor is given as the maximum of the real part $\Re f$ over the set of complex valued vectors of norm one.  The critical vector of $\Re f$ is an anti-eigenvector of $\bF=\frac{1}{d}\nabla f$. The critical value is the eigenvalue of this system. In this paper we first convert this eigenvalue problem to the anti-fixed point $\bF(\x)=\bar \x$. Next we show that this anti-fixed point is a fixed point of another polynomial map $\bH$, where $\bH(\x)=\overline{\bF(\overline{\bF(\x)})}$.   In the generic case, where the hypersurface $f(\x)=0$ is a smooth hypersurface in the complex projective space,  the number of fixed points is $(d-1)^{2n}$, counted with multiplicities.
In this case we find all fixed of $\bH$ and compute the $\|f\|_{\sigma,\C}$.  In the real case $\F=\R$, it is enough to consider the fixed points of $\bF$.
Our approach is completely different  from the standard optimization methods that are used now in the literature \cite{CHLZ12,LMV00,FMPS13,FT15,Nie14,Tong}.

We now highlight the new and most important results of our paper. Associate with $\cS\in\rS^d\C^n$ the polynomial $f(\x)=f_{\cS}(\x)=\cS\times\otimes^d\x,\x\in\C^n$.
(See the beginning of \S\ref{sec: spwcnrm} for tensor notations.)  Let $\mathbf{F}$ and $\bH$ be defined as above.
Each component of $\mathbf{F}$ and $\bH$ is a homogeneous polynomial of degree $(d-1)$ and $(d-1)^2$ respectively. Theorem \ref{theofoundthm} shows that the complex spectral norm of $\cS\in\rS^d\C^n$ can be computed by finding  the fixed points of the homogeneous polynomial map $\mathbf{H}:\C^n\to\C^n$, which is the polynomial system of equations $\bH(\x)-\x=\0$ .  Let $\rH(d,n)$ be the classical hyperdeterminant variety in the space $\rP(d,n,\C)$ \cite{GKZ}.  Assume that $f\in\rP(d,n,\C)\setminus \rH(d,n)$.  Then the number of fixed points of $\bF$ and $\bH$ is $(d-1)^n$ and $(d-1)^{2n}$ respectively, (counted with multiplicities).
 If in addition $\cS$ is real then its real spectral norm can be found by considering only the real fixed points of $\mathbf{F}$ in projective space $\bbP\C^n$.  

Recall the elimination method for finding the roots of a system of $n$ polynomial equations in $n$ variables, each of degree at most $d$, with only isolated roots  \cite{vdW, Mor}.  Its complexity upper bound $d^{2^{n}}$ follows from Kronecker's work \cite{Hei}.  If the system does not have roots at infinity then the   solutions of this system have rational univariate representation.   The arithmetic complexity of finding these roots is $O(d^{c\,n})$ for some $c>0$ \cite{Rou99,MST17}.
 Using the above result, and the fact that the maps $\textbf{F}$ or $\bH$ corresponding to $f_{\cS}\in\rP(d,n,\C)\setminus \rH(d,n)$  have a finite number of fixed points with a polynomial univariate representation, we give an algorithm to find the $\F$-spectral norm of $\cS\in\rS^d\F^n$ with an arbitrary relative precision with respect to the Hilbert-Schmidt norm of $\|\cS\|$.  See Theorems \ref{cesnsingcten}, \ref{cessingcten} and \ref{cesrealten}.  The arithmetic complexity of this algorithm is $O(d^{c\,n})$.
We remark that in all papers on complexity of finding the roots of zero dimensional zero set of polynomial equations in more than one variable cited in this paper we did not see any results on approximation of these roots within $\delta>0$ precision.

We study in detail the case of $d$-mode symmetric qubits, which are tensors in $\rS^d\C^2$ of Hilbert-Schmidt norm one.  We show that the nonzero fixed points of $\mathbf{H}$ can be computed by finding the roots of the corresponding
polynomial of one complex variable of degree at most $(d-1)^2+1$, provided that this symmetric qubit is not in the exceptional family, (Theorem \ref{computdqubspecnrm}).    We give a polynomial time algorithm for a relative approximation for all symmetric qubits, (Theorem \ref{complexaynqub}).  If  $\cS$ is real valued then its real spectral norm depends only on the real roots of this polynomial,
or actually, on the real root of another polynomial of degree at most $d+1$.  

Our results have an important application to the notion of the geometric measure of entanglement of $d$-partite symmetric states, (bosons),  in quantum physics and its computation.  Recall that a $d$-partite state is represented by a $d$-mode tensor $\cT$ of Hilbert-Schmidt norm one: $\|\cT\|=1$.  One of the most important notion in quantum physics is the entanglement of $d$-partite systems \cite{EPR35, Sch35,Sch36,Bel64}. A state $\cT$ is called entangled if it is not a product state, (rank one tensor).  The distance of a state $\cT$ to the product states is called the geometric measure of entanglement.  It is given by $\sqrt{2(1-\|\cT\|_{\sigma,\C})}$, where $\|\cT\|_{\sigma,\C}$ is the $\C$-spectral norm of $\cT$.  (See \S\ref{sec: spwcnrm}).
In particular, we deduce that the geometric measure of entanglement of a $d$-partite  symmetric state $\cS\in\rS^d\C^n$, called a symmetric $d$-\emph{qunit}, is polynomial time computable  in $d$ for a fixed $n$.  For symmetric qubits our results have much better complexity than in the case $n>2$.

We now survey briefly the contents of our paper. In \S\ref{sec: spwcnrm} we state our notations for tensors.  We recall the definition of the spectral norm of a tensor
$\cT$. We state the well known connection between the notion of the geometric measure of entanglement and the spectral norm of the $d$-partite state \eqref{defgment}. 

In \S\ref{sec:symten}  we first discuss the identification of $d$-symmetric tensors with the homogeneous polynomials of degree $d$.  Then we study the spectral norm of $d$-symmetric tensors on $\F^n$.  We recall the remarkable theorem of Banach \cite{Ban38} \eqref{Banchar} that characterizes the spectral norm of a symmetric tensor, which was rediscovered a number of times in the mathematical and physical literature \cite{CHLZ12,Fri13,Hubetall09}.   

In \S\ref{sec:critpts} we study the critical points of the homogeneous polynomial $f$ of degree $d$ on the unit sphere in $\F^n$.  We call a symmetric tensor $\cS$, where $f=f_{\cS}$, \emph{singular} if the system $\nabla f(\x)=\0$ has a nontrivial solution in $\C^n$.
Equivalently, if the corresponding hypersurface $f(\x)=0$ in the projective space $\mathbb{P}\C^n$ is singular.  
We show that the critical points of the real part of $f(\x)$ correspond to anti-fixed points of $\bF$ for $\F=\C$: $\bF(\x)=\bar\x$ and to fixed points of $\pm \bF$ for $\F=\R$. 
Theorem \ref{theofoundthm}, explained above, outlines our theoretical approach for numerical computation of the spectral norm of symmetric tensors.  Using the degree theory we give in Theorem \ref{estnumbanteig} lower and upper bounds on the number of  anti-fixed points of $\bF$ corresponding to nonsingular $\cS\in\rS^d\C^n$.  

In \S\ref{sec:dnqudit} we study the available algorithms and their complexities to approximate the spectral norms of symmetric tensors in $\rS^d\F^n$ for a fixed $n$.
A symmetric tensor $\cS$, and the corresponding $f_{\cS}\in\rP(d,n,\C)$, are called strongly nonsingular if they are nonsingular and the coordinates $x_1$ of the $(d-1)^{2n}$ fixed points of $\bH$ are pairwise distinct.  We show that most of the symmetric tensors in $\rS^d\C^n$ are strongly nonsingular.   The fixed points of $\bH$ corresponding to a strongly nonsingular tensor satisfy the conditions of the shape Lemma \cite{Rou99}.
In Theorem \ref{cesnsingcten}  we consider the case of a strongly nonsingular $\cT\in\rS^d\Z[\bi]^n$,where $\Z[\bi]$ are Gaussian integers.   Assume that the coordinates of $\cT$ are bounded in absolute value by $2^{\tau}$ for some $\tau\in\N$.  Using recent results in \cite{BFS15} and \cite{MST17} we show that the bit complexity of the computation of an approximation $L(\cT)$ to the norm $\|\cT\|_{\sigma,\C}$ with a relative precision $2^{-e}, e\in\N$ is  $\tilde O\big((\tau+e) d^{8n}\big)$.

In Theorem \ref{cessingcten} we discuss the computational complexity of an approximation $L(\cT)$ to a given $\cT\in\rS^d\Z[\bi]^n$, without assuming that $\cT$ is strongly nonsingular.  We give a probabilistic algorithm to compute $L(\cT)$ with a similar bit complexity.
Similar results are obtained in Theorem \ref{cesrealten}  for an approximation of the spectral norm of $\cT\in\rS^d\Z^n$ with slightly better complexity estimations.
Theorems \ref{theofoundthm}, \ref{cesnsingcten}, \ref{cessingcten},  and \ref{cesrealten}  constitute the first major contribution of this paper.

An obvious question is what is the complexity of finding an approximation $L(\cT)$ to $\cT\in\rS^d\Z^n$ if we do not keep $n$ fixed. 
Theorem \ref{NPhardquart} shows that an approximation of the spectral norm of homogeneous quartic polynomials with an arbitrary precision is NP-hard.  This result follows from the old result of Motzkin-Straus \cite{MS65} relating the clique number of a graph to a certain maximum problem for the adjacency matrix of the graph, and its tensor interpretation to the spectral norm of tensors  \cite[(8.2)]{FL18}.  Unfortunately,
it was not observed in \cite{FL18} that the corresponding tensor is symmetric.
Note that the approximation algorithm outlined in Theorem \ref{cesrealten} has at least complexity $O(3^{n^2})$, while the brute force method looking over all possible subsets of $n$ vertices of the given graph is $O(2^n)$. 

In \S\ref{sec:dqubit} we discuss in detail theoretical and numerical aspects of the computation of the spectral norm of $\cS\in\rS^d\F^2$.  In Theorem \ref{computdqubspecnrm} we show that
the fixed points of the corresponding $\bH$ in this case can be reduced 
to one polynomial equation of degree at most $(d-1)^2+1$, unless we are in the exceptional family.  In the nonexceptional case we give a simple formula to compute the spectral norm.  In Theorem \ref{complexaynqub} we give an approximation algorithm for the spectral norm of $\cT\in\rS^d\Z[\bi]^2 $ with a relative  error $2^{-e}$ of bit complexity $\tilde\cO(d^2(d^4\max(d^2,\tau)+e))$.
 For $\cT\in\rS^d\Z^2$ we have better complexity results.  Theorems \ref{computdqubspecnrm} and \ref{complexaynqub} constitute the second main contribution of this paper.

 In \S\ref{sec:excepcase} we analyze completely the exceptional family.   We show how to obtain a relative approximation for symmetric tensors in this family.  The complexity of this approximation is the same as for the nonexceptional family.
 
In \S\ref{sec:numerex} we give numerical examples of our method for calculating the complex and the real  spectral norm of some of $\cS\in\rS^d\C^n$ for $n=2,3,4$.  Many of our examples correspond to polynomials that are sum of two monomials.  Lemma \ref{sum2mon} shows that in this case one can assume that the coefficients of these two monomials are nonnegative, if we compute the complex spectral norm.  That is, the corresponding symmetric tensors have nonnegative entries.  (This result is false if we consider the real spectral norm of real symmetric tensor which corresponds to a sum of two monomials.)  Most of the examples of $d$-qubits considered in \cite{AMM10} are sums of two monomials. The authors believe that their examples for $d=4,5, \dots,12$ are the most entangled $d$-qubits.  Our software confirms the values of the spectral norms of the examples in \cite{AMM10}.  We also consider five one complex parameter families of these examples, and we compute a number of values of the spectral norms in these families.  As expected, in all these computed examples the spectral norms are higher than in the examples in \cite{AMM10}.

In Appendix 1 we consider a standard orthonormal basis in $\rS^d\C^n$,  the analog of Dicke states in $\rS^d\C^2$ \cite{Dic}, and the entanglement of each element in the basis.  We give an upper bound on the entanglement of symmetric states in $\rS^d\C^n$.

In Appendix 2 we discuss the complexity results associated with a system of $m$ polynomial equations in $m$ variables with isolated roots and no roots at infinity. 
We recall simple necessary and sufficient conditions on such systems.  We define  an $x_1$-simple system which has only simple solutions with pairwise distinct $x_1$ coordinates.  For $x_1$-simple systems the reduced  Gr\"obner basis with respect to the order $x_1 \prec\cdots \prec x_m$ satisfies the shape lemma.
We recall the known complexity results of finding  the reduced Gr\"obner basis
in this case \cite{BFS15}.  We also recall the complexity results for finding the roots of 
a polynomial in one complex variable \cite{NR96}.  Lemma 5 summarize the complexity results of finding all the roots of $x_1$-simple system with a given precision.

In Appendix 3 we discuss briefly the Majorana representation.  This representation is used in physics literature \cite{AMM10,MGBB10}.  We explain how Majorana representation suggests two kinds of the most entangled $d$-symmetric qubits, which solve either T\'oth’s or Thomson's problems \cite{Why52,Tho04}.  Most of the example in \cite{AMM10} are based on these two problems.
However, in certain cases as shown in \cite{AMM10} the most entagled symmetric states do not solve neither of the above problems.

\section{Spectral norm and entanglement}\label{sec: spwcnrm}
For a positive integer $d$, i.e., $d\in\N$, we denote by $[d]$ the set of consecutive integers $\{1,\ldots,d\}$.
Let $\F\in\{\R,\C\}$, $\n=(n_1,\ldots,n_d)\in\N^d$.  We will identify the tensor product space
$\otimes_{i=1}^d \F^{n_i}$ with the space of $d$-arrays $\F^{\mathbf{n}}$.
The entries of $\cT\in\F^{\mathbf{n}}$ 
are denoted as $\cT_{i_1,\ldots,i_d}$.  We also will use the notation $\cT=[\cT_{i_1,\ldots,i_d}]$.
So $\cT$ is called a vector for $d=1$, a matrix for $d=2$ and a tensor for $d\ge 3$.  Note that the dimension of $\F^{\mathbf{n}}$ is $N(\n)=n_1\cdots n_d$. 

Assume that $d\ge 2$ is an integer and $k\in \{0,1\}$.  For $\n=(n_1,\ldots,n_d)\in\N^d$ let  $\mathbf{m}=(n_{k+1},\ldots,n_d)\in \N^{d-k}$.  Assume that $\cT\in \F^\n$ and $\cS\in\F^{\m}$. Then  $\cT\times \cS$ is the scalar $\sum_{i_1=\cdots=i_d}^{n_1,\ldots,n_d}\cT_{i_1,\ldots,i_d}\cS_{i_1,\ldots,i_{d}}$ for $k=0$ and a vector in $\F^{n_1}$, whose $i$-th coordinate  is given by
$(\cT\times \cS)_{i}=\sum_{i_{2}=\cdots=i_d=1}^{n_2,\ldots,n_d} \cT_{i,i_2,\ldots,i_d}\cS_{i_{2},\ldots,i_{d}}$, for $k=1$.

The inner product on $\F^{\mathbf{n}}$ is given as $\langle\cS,\cT\rangle:=\cS\times \overline{\cT}$, where $\overline{\cT}=[\overline{\cT_{i_1,\ldots,i_d}}]$.
Furthermore, $\|\cS\|=\sqrt{\langle\cS,\cS\rangle}$ is the Hilbert-Schmidt norm of $\cS$.  Assume that $\x_i=(x_{1,i},\ldots,x_{n_i,i})\trans\in\F^{n_i}$ for $i\in[d]$.
Then $\otimes_{i=1}^d \x_i$ is a tensor in $\F^{\mathbf{n}}$, with the entries $(\otimes_{i=1}^d \x_i)_{i_1,\ldots,i_d}=x_{i_1,1}\cdots x_{i_d,d}$.
($\otimes_{i=1}^d \x_i$ is called a rank one tensor if all $\x_i\ne \0$.)  Assume that $\x_1=\cdots=\x_d=\x$.  Then $\otimes^d\x=\otimes_{i=1}^d \x_i$.

Denote the unit sphere in $\F^n$ by $\rS(n,\F)=\{\x\in\F^n, \|\x\|=1\}$. Recall that the spectral norm of $\cT\in \F^{\mathbf{n}}$ is given as
\begin{eqnarray*}\label{specnrmdef}
\|\cT\|_{\sigma,\F}=\max\{|\cT\times \otimes_{i=1}^d \x_i|, \; \x_i\in \rS(n_i,\F) \textrm{ for } i\in[d]\}.
\end{eqnarray*}

Assume that $d=2$.  Then $\cT$ is a matrix $T\in\F^{n_1\times n_2}$. In that case
$\|T\|_{\sigma,\F}$ is the spectral norm of $T$, and is equal to its maximum singular value $\sigma_1(T)$.  In particular, for $T\in \R^{n_1\times n_2}$ one has equality $\|T\|_{\sigma,\R}=\|T\|_{\sigma,\F}$.  Furthermore, as $\sigma_1(T)^2$ is the maximum eigenvalue of hermitian matrix $T T^*$ or $T^*T$.   It is well known that $\sigma_1(T)$ can be computed in polynomial time in the entries of $T$ and $\max(n_1,n_2)$.  See for example \cite{GV13} for a general rectangular matrix, or \cite{KW92} for direct of Lancos algorithm for $T T^*$.  Another method for $T$, with Gaussian integer entries, is as follows: First compute the characteristic polynomial $p(z)$ of the hermitian matrix $H(T)=\left[\begin{array}{cc}0 &T\\T^*&0\end{array}\right]$.  The complexity of such an algorithm is described in \cite{DPW05}.  Next recall that the eigenvalues of $H(T)$ are $\pm \sigma_i(T)$ and $0$ \cite[Theorem 4.11.1]{Frb16}.  Now use well known algorithms as \cite{NR96} to approximate the roots $p(z)$.  

In the rest of this paper we assume that $d\ge 3$, i.e., $\cT$ is a tensor, unless stated otherwise.
Unlike in the matrix case, for a real tensor $\cT\in \R^{\mathbf{n}}$ it is possible that $\|\cT\|_{\sigma,\R}<\|\cT\|_{\sigma,\C}$ \cite{FL18}.
For simplicity of notation we will let $\|\cT\|_{\sigma}$ denote $\|\cT\|_{\sigma,\C}$, and no ambiguity will arise.

A  standard way to compute $\|\cT\|_{\sigma,\F}$ is  an alternating maximization with respect to one variable, while other variables are fixed, see \cite{LMV00}.  Other variants of this method is maximization on two variables using the SVD algorithms \cite{FMPS13}, or the
Newton method \cite{FT15,Tong}.  These methods in the best case yield a convergence to a local maximum, which provide a lower bound
to $\|\cT\|_{\sigma,\F}$.  Semidefinite relaxation methods, as in \cite{Nie14}, will yield an upper bound to $\|\cT\|_{\sigma,\F}$, which will converge in some cases to $\|\cT\|_{\sigma,\F}$.

Recall that in quantum physics $\cT\in\C^{\mathbf{n}}$ is called a state if $\|\cT\|=1$.  (Furthermore, all tensors of the form $\zeta\cT$, where $\|\cT\|=1$ and 
$\zeta\in\C,|\zeta|=1$ are identified as the same state.  
That is, the space of the states in $\C^{\mathbf{n}}$ is the quotient space $\rS(N(\n),\C)/\rS(1,\C)$.  For simplicity of our exposition will ignore this identification.)
Denote by $\Pi^{\mathbf{n}}$ the product states in $\C^{\mathbf{n}}$:
\[\Pi^{\mathbf{n}}=\{\otimes_{i=1}^d \x_i,\; \x_i\in\rS(n_i,\C),i\in[d]\}.\]

The geometric measure of entanglement of a state $\cT\in\C^{\mathbf{n}}$ is 
\begin{eqnarray}\label{defgment}
\mathrm{dist}(\cT,\Pi^{\mathbf{n}})=\min_{\cY\in\Pi^{\mathbf{n}}} \|\cT-\cY\|.
\end{eqnarray}
As $\|\cT\|=\|\cY\|=1$ it follows that $\mathrm{dist}(\cT,\Pi^{\mathbf{n}})=\sqrt{2(1-\|\cT\|_{\sigma})}$.  Hence an equivalent measurement of entanglement is \cite{GFE09}
\begin{equation}\label{defetaT}
\eta(\cT)=-\log_2 \|\cT\|_\sigma^2.
\end{equation}
The maximal entanglement is
\begin{equation}\label{maxentn}
\eta(\n)=\max_{\cT\in\C^{\mathbf{n}}, \|\cT\|=1} -\log_2\|\cT\|_\sigma^2.
\end{equation}
See \cite{DFLW17} for other measurements of entanglement using the nuclear norm of $\cT$.  Lemma 9.1 in \cite{FL18} implies
\[\eta(\n)\le \log_2 N(\n).\]

Let $n^{\times d}=(n,\ldots,n)\in\N^d$.
For $n=2$ we get that $\eta(2^{\times d})\le d$.  In \cite{Jungetall08} it is shown that $\eta(2^{\times d})\le d-1$.  A complementary result is given in \cite{GFE09}:  For the set of states of Haar measure at least $1-e^{-d^2}$ on the sphere $\|\cT\|=1$ in $\otimes^d\C^{2}$ the inequality $\eta(\cT)\ge d - 2\log_2 d-2$ holds. A generalization of this result to $\otimes^d\C^{n}$ is given in \cite{DM18}.
\section{Symmetric tensors}\label{sec:symten} 
A tensor $\cS=[\cS_{i_1,\ldots,i_d}]\in\otimes^d\F^n$ is called symmetric if $\cS_{i_1,\ldots,i_d}=
\cS_{i_{\omega(1)},\ldots,i_{\omega(d)}}$ for every permutation $\omega:[d]\to[d]$.  
Denote by $\rS^d\F^n\subset \otimes^d\F^n$ the vector space of $d$-mode symmetric tensors on $\F^n$.
In what follows we assume that $\cS$ is a symmetric tensor and $d\ge 2$, unless stated otherwise.  A tensor
$\cS\in\rS^d\F^n$ defines a unique homogeneous polynomial of degree $d$ in $n$ variables $f(\x)=\cS\times\otimes^d\x$, where
\begin{eqnarray*}\label{defpolfx}
f(\x)=\sum_{ j_k+1\in [d+1],k\in[n], j_1+\cdots +j_n=d} \frac{d!}{j_1!\cdots j_n!} f_{j_1,\ldots,j_n} x_1^{j_1}\cdots x_n^{j_n}.
\end{eqnarray*}
Conversely, a homogeneous polynomial $f(\x)$ of degree $d$ in $n$ variables defines a unique symmetric $\cS\in\rS^d\F^n$ as given in part (4) of Lemma \ref{isolem}.
 
Hence it is advantageous to replace $\rS^d\F^n$ by the isomorphic space of all homogeneous polynomials of degree $d$ in $n$ variables over $\F$, denoted as $\rP(d,n,\F)$.  We now introduce the standard multinomial notation as in \cite{Rez92}.
Let $\Z_+$ be the set of all nonnegative integers. Denote by $J(d,n)$ be the set of all $\bj=(j_1,\ldots,j_n)\in\Z_+^n$ appearing in the above definition of $f(\x)$:
\begin{eqnarray*}\label{defJdn} 
J(d,n)= \{\bj=(j_1,\ldots,j_n)\in\Z_+^n,\; j_1+\cdots+j_n=d\}.
\end{eqnarray*}
It is well known that $|J(d,n)|={n+d-1\choose n}={n+d-1\choose d-1}$, see for example \cite{Rez92}.	
Define 
\begin{eqnarray*}\label{defcbj}
c(\bj)=\frac{d!}{j_1!\cdots j_n!}
\end{eqnarray*}
For $\x=(x_1,\ldots,x_n)\trans\in\F^n$ and $\bj=(j_1,\ldots,j_n)\in J(d,n)$ let $\x^{\bj}$ be the monomial $x_1^{j_1}\cdots x_n^{j_n}$.  Then the above definition of $f(\x)$  is equivalent to
\begin{equation}\label{defpolx1}
f(\x)=\sum_{\bj\in J(d,n)} c(\bj)f_{\bj}\x^{\bj}.
\end{equation}

Let $\bU$ and $\bV$ be finite dimensional vector spaces over $\F$.  Then $\bU$ and $\bV$ are isomorphic if and only if $\bU$ and $\bV$ have the same dimension.  
Assume that $\bU$ and $\bV$ are two inner product vector spaces over $\F$ of the same dimension $N$.  Then $L:\bU\to \bV$ is called an isometry if $L$ preserves the inner product.  Assume that $\bu_1,\ldots,\bu_N$ in $\bU$  is an orthonormal basis in $\bU$.  Then a linear transformation $L:\bU\to \bV$ is an isometry if and only if $\bv_1=L(\bu_1),\ldots,\bv_N=L(\bu_N)$ is an orthonormal basis in $\bV$.

The following lemma summarizes the properties of the  isomorphisms of $\rS^d\F^n$ to $\rP(d,n,\F)$ and to the auxiliary vector space $\F^{J(d,n)}$, and recalls Banach's characterization of the spectral norm of $\cS\in\rS^d\F^n$ as the maximum of the absolute value of the polynomial  $\cS\times(\otimes^d\x)$ on the unit sphere $\rS(n,\F)$ \cite{Ban38}: 
\begin{lemma}\label{isolem}
Let $\F^{J(d,n)}$ be the space of all vectors $\f=(f_{\bj}), \bj\in J(d,n)$.  Assume that the inner product and the Hilbert norm on $\F^{J(d,n)}$ are given by 
\begin{eqnarray*}\label{inprodFJ}
\langle \f,\bg\rangle=\sum_{\bj\in J(d,n)} c(\bj)f_{\bj}\overline{g_{\bj}},\quad \|\f\|=\sqrt{\langle \f,\f\rangle},\quad \f=(f_{\bj}),\bg=(g_{\bj})\in \F^{J(d,n)}.
\end{eqnarray*}
Then
\begin{enumerate}
\item Let $\e_{\bj}=(\delta_{\bj,\bk})_{\bk\in J(d,n)}, \bj\in J(d,n)$, where $\delta_{\bj,\bk}$ is Kronecker's delta function, be the standard basis in $\F^{J(d,n)}$.  Then $\frac{1}{\sqrt{c(\bj})}\e_{\bj}, \bj\in J(d,n)$ is an orthonormal basis in $\F^{J(d,n)}$. 
\item $\F^{J(d,n)}$ is isomorphic to $\F^{n+d-1\choose d-1}$.  There is an isometry $L:\F^{J(d,n)}\to \F^{n+d-1\choose d-1}$ which maps the orthonormal basis $\frac{1}{\sqrt{c(\bj})}\e_{\bj}, \bj\in J(d,n)$ to the standard orthonormal basis in $\F^{n+d-1\choose d-1}$.
\item $\F^{J(d,n)}$ is isomorphic to $\rP(n,d,\F)$, where $\f$ corresponds to $f(\x)$ given by \eqref{defpolx1}.
\item The map $L:\rS^d\F^n\to\rP(d,n,\F)$ which is given by $L(\cS)=f$, where $f(\x)=S\times(\otimes ^d\x)$, is an isomorphism and an isometry. 
\item 
Assume that $\cS\in\rS^d\F^n$ and let  $\bF(\x)=\cS\times(\otimes^{d-1}\x)$.  Then
\begin{eqnarray}\label{Fformulas}
	&&\bF(\x)=(F_1(x),\ldots,F_n(\x))=\frac{1}{d}\nabla f(\x)=\frac{1}{d}(\frac{\partial f}{\partial x_1}(\x),\ldots,\frac{\partial f}{\partial x_n}(\x)),\\
	\label{Eulerform}
	&&\sum_{i=1}^n x_iF_i(\x)=f(\x).
	\end{eqnarray}
\item The spectral norm of a symmetric tensor is given by Banach's characterization \cite{Ban38}:
\begin{equation}\label{Banchar}
\|\cS\|_{\sigma,\F}=\max\{\frac{|\cS\times(\otimes^d\x)|}{\|\x\|^d}, \;\x\in\F^n\setminus\{\0\}\}=\max\{|\cS\times(\otimes^d\x)|, \;\x\in\rS(n,\F)\}, \quad \cS\in\rS^d\F^n.
\end{equation}
\end{enumerate} 
\end{lemma}
\begin{proof} Parts (1)-(3) are straightforward. 

\noindent Part (4).  Let $\cS\in\rS^d\F^n$.  Define $L(\cS)=f$, where $f(\x)=\cS\times (\otimes^d \x)$.  Clearly $f\in\rP(d,n,\F)$.  Assume that $f(\bx)$ is given by \eqref{defpolx1}.  For a given $i_1,\ldots,i_d\in[n]$ and $k\in[n]$ let $j_k$ be the number of times $k$ appears in the multiset $\{i_1,\ldots,i_d\}$.  Set $\bj=(j_1,\ldots,j_n)$.  Then $L(\cS)_{\bj}=f_{\bj}$.
It is straightforward to show that $L$ is an isomorphism. Furthermore, $\langle \cS,\cT\rangle=\langle L(\cS),L(\cT)\rangle$.  Hence $L$ is an isometry.

\noindent
Part (5).
Observe that $\frac{\partial\;\;\;\;}{\partial x_m} (\cS_{i_1,\ldots,i_d}\x^{\bj})$ is not zero if and only if $i_l=m$ for some $l\in[d]$.   So assume that $i_l=m$.  Since $\cS$ is symmetric we can choose $l\in[d]$. Hence $d F_m=\frac{\partial f\;\;\;}{\partial x_m}$, and  \eqref{Fformulas} holds.  Equality \eqref{Eulerform} is Euler's formula for $f\in\rP(d,n,\F)$ and $\bF=\frac{1}{d}\nabla f$.

Part (6) is Banach's theorem \cite{Ban38}, see \cite{FL18} for details.
\end{proof}

Banach's theorem \eqref{Banchar} was rediscovered several times since 1938.  In quantum physics literature it appeared in \cite{Hubetall09} for the case $\F=\C$.  In mathematical literature, for the case $\F=\R$, it appeared in \cite{CHLZ12,Fri13}.  (Observe that  a natural generalization of Banach's theorem to partially symmetric tensors is given in \cite{Fri13}.)

In view of Lemma \ref{isolem} it makes sense to introduce the spectral norm on $\F^{J(d,n)}$ and $\rP(d,n)$:
\begin{equation}\label{polspecnorm}
\|\ff\|_{\sigma,\F}=\|f\|_{\sigma,\F}=\max\{\frac{|f(\x)|}{\|\x\|^d},\x\in\F\setminus\{\0\}\}=\max\{|f(\x)|,\x\in\rS(n,\F)\},
\end{equation}
where $f(\x)$ is given by \eqref{defpolx1}.

We denote by $\cS(\bj)\in\rS^d\C^n$ the symmetric state corresponding to $\frac{1}{\sqrt{c(\bj})}\e_{\bj}, \bj\in J(d,n)$.  Note that $\cS(\bj)$ corresponds to the monomial $\sqrt{c(\bj)} \x^{\bj}$, i.e., $\cS(\bj)\times\otimes^d\x=\sqrt{c(\bj)} \x^{\bj}$.  We also let $\|f\|=\|\ff\|$.
\section{Critical points of $\Re (\cS\times\otimes^d \x)$ on $\rS(n,\F)$ }\label{sec:critpts}
Recall that the projective space $\bbP\C^N$ is the space of lines through the origin in  $\C^N$.  Thus a point in $\bbP\C^N$ is the equivalence class $[\by]=\{t\by, t\in\C\setminus\{0\}, \by\in\C^N\setminus\{\0\}\}$. Note that for each $[\by]\in\bbP\C^N$, the hyperplane $\{\bx\in\C^N,\;\langle \bx,\by\rangle=0\}$ is independent of the representative corresponding to $[\by]$.

Recall that $\cS\in\rS^d\C^n$ is called nonsingular \cite{FO14} if 
\[\cS\times\otimes^{d-1}\x=\0\Rightarrow \x=\0.\]
Otherwise $\cS$ is called singular.
A nonzero homogeneous polynomial $f(\x)$ defines a hypersurface $H(f):=\{\x\in\C^{n}\setminus\{\0\},f(\x)=0\}$ in $\bbP\C^n$.  
$H(f)$ is called a smooth hypersurface if $\nabla f(\x)\ne \0$ for each $\x\ne \0$ that satisfies $f(\x)=0$.
\begin{proposition}\label{relatfxS}  Assume that $\cS\in\rS^d\C^n$.
	Let $f(\x)=\cS\times\otimes^d\x$.  Then
	$\cS$ is nonsingular if and only if $H(f)$ is a smooth hypersurface in $\bbP\C^n$.
\end{proposition}
\begin{proof} Let $\bF=\frac{1}{d}\nabla f$.  Assume that $\bF(\x)=\0$ for some $\x\ne \0$.  Euler's identity yields that $f(\x)=0$.  Use part (5) of Lemma \ref{isolem} to deduce the proposition.
\end{proof}
The following result is well known \cite{GKZ}:
\begin{proposition}\label{hypdetvar}  Denote by $\bbP\C^{J(d,n)}$ the complex projective space corresponding to the affine space $\C^{J(d,n)}$.  With each $[\ff]\in \bbP\C^{J(d,n)}$ associate the hypersurface  $f(\x)=0$ in $\bbP\C^n$, where $f(\x)$ is given by \eqref{defpolx1}.  Then the set of singular hypersurfaces is the hyperdeterminant variety $\rH(d,n)\subset \bbP\C^{J(d,n)}$, which is the zero set of the hyperdeterminant polynomial on $\C^{J(d,n)}$. 
\end{proposition}
\begin{corollary}\label{hypdetvar} The set of singular symmetric tensors in $\rS^d\C^n$ is the zero set $V(d,n)\subset \rS^d\C^n$ of the polynomial on $\rS^d\C^n$ which is induced by the hyperdeterminant polynomial on  $\C^{J(d,n)}$.
\end{corollary}

For a complex number $z=x+\bi y\in\C, x,y\in\R$ we denote $x=\Re z$ and $y=\Im z$.
Fix $\x\in\C^n$ and let $\zeta\in \C$.  Then $\cS\times\otimes^d\left(\zeta\x\right)=\zeta^d\left(\cS\times\otimes^d\x\right)$.  Hence there exists $\zeta\in\C, |\zeta|=1$ such that
$|\cS\times\otimes^d\x|=\Re\left(\cS\times\otimes^d\left(\zeta\x\right)\right)$.
Therefore for  $\F=\C$ we can replace the characterization \eqref{Banchar} with:
\begin{eqnarray*}\label{maxcharspecnrmc}
\|\cS\|_{\sigma}=\max_{\x\in\rS(n,\C)} \Re(\cS\times\otimes^d\x), \textrm{ for } \cS\in\rS^d\C^n.
\end{eqnarray*}
Let $f\in \rP(d,n,\F)$.   Consider $\Re f|\rS(n,\F)$,  the restriction of $\Re f$ to the sphere $\rS(n,\F)$.  Suppose first that 
$\F=\R$.  Then $\Re f|\rS(n,\R)=f|\rS(n,\R)$.  A point $\x\in\rS(n,\R)$ is called a critical point of $f|\rS(n,\R)$ if the directional derivative of $f$ at $\x$ in direction of each $\y$, where $\langle \y,\x\rangle =0$, is $0$.  Assume that $\x$ is a critical point of $f|\rS(n,\R)$.  Then $f(\x)$ is called a critical value of $f|\rS(n.\R)$.  

Assume now that $\F=\C$.  View $\C^n$ as $\R^{2n}$ by writing $\z\in\C^n$ as $\z=\x+\bi\y$, where $\x,\y\in\R^n$.  Clearly $\Re f(\z), \z\in\C^n$ can be viewed as a homogeneous polynomial $g(\x,\y)$ of degree $d$ in $2n$ variables $(\x,\y)$. 
We then identify $\rS(n,\C)$ with $\rS(2n,\R)$.  Then $\Re f|\rS(n,\C)$ is $g|\rS(2n,\R)$.  Thus a critical point of $\Re f|\rS(n,\C)$ is the critical point of $g|\rS(2n,\R)$, and the critical value of $\Re f|\rS(n,\C)$ is the critical value of $g|\rS(2n,\R)$.
Note that the points $\x\in\rS(n,\F)$ where $\Re f|\rS(n,\F)$ is maximum or minimum
are critical points, and the maximum and minimum values are critical values.
\begin{lemma}\label{critptlem}  Assume that $\cS\in\rS^d\F^n, d\ge 2$.  
	A point $\x\in\rS(n,\F)$ is a critical point of $\Re f(\x)$ on $\rS(n,\F)$
	if and only if
	\begin{equation}\label{critpt}  
	\cS\times \otimes^{d-1}\x=\lambda\overline{\x}, \quad \x\in\rS(n,\F), \lambda\in\R,
	\end{equation}
	where $\overline{\x}$ denote the complex conjugate of $\x$.   The number of critical values $\lambda$ satisfying (\ref{critpt}) is finite.
\end{lemma}
\begin{proof} 
	First assume that $\F=\R$.  Let $\x\in\rS(n,\R)$.  First assume that $\x$ is a critical point of $f(\z)=\cS\times\otimes^d\z$ for $\z\in\rS(n,\R)$.  Let $\y\in\R^n$ be orthogonal to $\x$: $\y\trans \x=0$.  Then $\|\x+t\y\|=\sqrt{1+t^2\|\y\|^2}=1+O(t^2)$ for $t\in\R$.  Clearly
	\[\cS\times \otimes ^d (\x+t\y)=\cS\times \otimes^d \x+td\y\trans (\cS\times \otimes^{d-1}\x) +O(t^2).\]
	As $\x$ is a critical point of $\cS\times\otimes^d\z$ for $\z\in\rS(n,\R)$ it follows that $\y\trans (\cS\times \otimes^{d-1}\x)=0$ for each $\y$ orthogonal to $\x$.
	Hence $\cS\times \otimes^{d-1}\x$ is colinear with $\x$.  As $\bar\x=\x$ for each $\x\in\R^n$ we deduce (\ref{critpt}).  Similar arguments show that if (\ref{critpt}) holds for
	$\x\in\rS(n,\R)$ then $\x$ is a critical point.  
	
	As $f(\x)$ is a polynomial on $\R^n$ it follows that 
	the set of critical points of $f|\rS(n,\R)$ is a real algebraic set.   This algebraic set is a finite union of connected algebraic sets \cite[Proposition 1.6]{Cos05}.  On each connected algebraic set of critical points $f$ is a constant function, whose value on this set is a critical value.  This proves that the number of critical values for $\F=\R$ is finite.
	
	Second assume that $\F=\C$.  View $\C^{n}$ as $2n$-dimensional real vector space $\R^{2n}$ with the standard inner product $\Re (\y^*\x)$, where $\y^*=\overline{\y}\trans$.  Hence $\|\x\|=\sqrt{\Re (\x^*\x)}$.  Assume that $\x\in\rS(n,\C)$
	is a critical point of $\Re (\cS\times \otimes^d\z)$ on $\rS(n,\C)$.  Let $\y\in\C^n$ be orthogonal to $\x$: $\Re (\y^*\x)=\Re(\overline{ \y}\trans \x)=0$.  Then $\|\x+t\y\|=\sqrt{1+t^2\|\y\|^2}
	=1+O(t^2)$ for $t\in\R$.   Hence
	\begin{eqnarray}\label{varforRSx}
	\Re(\cS\times \otimes ^d (\x+t\y))=\Re(\cS\times \otimes^d \x)+td\Re(\y\trans (\cS\times \otimes^{d-1}\x)) +O(t^2).
	\end{eqnarray}
	As $\x$ is a critical point we deduce that 
	\[0=\Re(\y\trans (\cS\times \otimes^{d-1}\x))=\Re(\y^* (\overline{\cS\times \otimes^{d-1}\x})).\]
	Hence $\overline{\cS\times \otimes^{d-1}\x}$ is $\R$-colinear with $\x$.  Thus (\ref{critpt}) holds.  Vice versa, suppose that \eqref{critpt} holds.  As $\lambda\in\R$ and  $0=\Re (\y^*\x)=\Re (\y\trans \bar\x)=0$  the equality \eqref{varforRSx} yields that $\x$ is a critical point.
	
	Since  $q(\x):=\Re(\cS\times \otimes^d\x)$ is a polynomial on $\C^n\sim\R^{2n}$ it follows from the above arguments for $\F=\R$ that
	$q|\rS(n,\C)$ has a finite number of critical values.
\end{proof}

Clearly, a maximum point of $|\cS\times \otimes^d\x|$ on $\rS(n,\F)$ is a critical point of $\Re\left(\cS\times\otimes^d\x\right)$ on $\rS(n,\F)$.  Hence:
\begin{corollary}\label{maxeig}  Let $d,n\ge 2$  be integers.  
\begin{enumerate} 
		\item Assume that $\cS\in\rS^d\R^n$.  Then there exists $\x\in\rS(n,\R)$ satisfying (\ref{critpt}) such that $|\lambda|=\|\cS\|_{\sigma,\R}$.  Furthermore, $\|\cS\|_{\sigma,\R}$ is the maximum of all $|\lambda|$ satisfying (\ref{critpt}).
		\item Assume that $\cS\in\rS^d\C^n$.  Then there exists $\x\in\rS(n,\C)$ satisfying (\ref{critpt}) such that $\lambda=\|\cS\|_{\sigma}$.  Furthermore, $\|\cS\|_{\sigma}$ is the maximum of all $|\lambda|$ satisfying (\ref{critpt}).
\end{enumerate}
\end{corollary}

We call $\x\in\rS(n,\F)$ and $\lambda\in\F$ an eigenvector and an eigenvalue of $\cS\in\rS^d\F^n$ if the following conditions hold \cite{CS}:
\begin{eqnarray}\label{defeigvl}
\cS\times\otimes^{d-1}\x=\lambda\x, \quad \x\in\rS(n,\F),\;\lambda\in \F, \quad \cS\in\rS^d\F^n.
\end{eqnarray}

Assume that $\F=\R$.  Then the above equality is equivalent to (\ref{critpt}).  First  assume that $d$ is odd and $\x$ is an eigenvector of $\cS$.  Then $-\x$ is an eigenvector of $\cS$ corresponding to $-\lambda$.  Hence without loss of generality we can consider only nonnegative eigenvalues of $\cS$.  Second assume that $d$ is even and $\x$ is an eigenvector of $\cS$.  Then $-\x$ is also eigenvector of $\cS$ corresponding to $\lambda$.

A vector $\x\in\rS(n,\C)$ and a scalar $\lambda\in\R$ that satisfy (\ref{critpt}) are called the \emph{anti-eigenvector} and \emph{anti-eigenvalue} of $\cS\in\rS^d\C^n$.
Note that if $\x$ is an anti-eigenvector and $\lambda$ a corresponding anti-eigenvalue then $\zeta\x$ is also anti-eigenvector with a corresponding anti-eigenvalue $\varepsilon\lambda$,  where $\varepsilon=\pm 1$ and $\zeta^d=\varepsilon$.  Hence, we can always assume that each nonzero anti-eigenvalue is positive, and there are $d$ different choices of $\zeta$ such that $\zeta\x\in\mathrm{span}(\x)$ is an anti-eigenvector corresponding to a given positive anti-eigenvalue $\lambda$.

We now state the first main result of this paper, which gives the theoretical foundation for the computational methods of our paper.
\begin{theorem}\label{theofoundthm}  Let $\cS\in\rS^d\F^n\setminus\{0\}$ and $d\ge 3$.  Associate with $\cS$ the polynomial $f(\x)=\cS\times \otimes^d\x\in\rP(d,n,\F)\setminus\{0\}$. Let $\bF=\frac{1}{d}\nabla f$.  Denote 
\begin{eqnarray*}
\mathrm{fix}(\bF)=\{\x\in\C^n, \;\bF(\x)=\x\}, \quad
\mathrm{afix}(\bF)=\{\x\in\C^n, \;\bF(\x)=\bar \x\}, 
\end{eqnarray*}
the set of fixed and antifixed points of $\bF$ in $\C^n$ respectively.  
Let $\omega_{d-2}=e^{\pi\bi/(d-2)}$, ( $\omega_{d-2}^{d-2}=-1$).
Assume that $f\in\rP(d,n,\R)$.   Denote 
\begin{eqnarray*}
\mathrm{fix}_{\R}(\bF)=
\begin{cases}
\mathrm{fix}(\bF)\cap\R^n \textrm{ if } d \textrm{ is odd},\\
(\mathrm{fix}(\bF)\cup \bar\omega_{d-2}\mathrm{fix}(\bF))\cap \R^n \textrm{ if } d \textrm{ is even}.
\end{cases}
\end{eqnarray*}
\begin{enumerate}
\item Assume that $\F=\R$.  Then 
\begin{eqnarray}\label{Sinftynrmformr}
\|\cS\|_{\sigma,\R}=\|f\|_{\sigma,\R}=\max\{\frac{|f(\x)|}{\|\x\|^d}, \x\in\mathrm{fix}_{\R}(\bF)\setminus\{\0\}\}.
\end{eqnarray}
\item Let $\F=\C$.  Denote by $\overline{\bF}:\C^n\to\C^n$ the polynomial mapping given by the equality $\overline{\bF}(\x)=\overline{\bF(\bar\x)}$.  Let $\bH=\overline{\bF}\circ \bF$.
Then the set of fixed points of $\mathrm{fix}(\bH)=\{\x\in\C^n, \bH(\x)=\x\}$ contains $\mathrm{afix}(\bF)$.  Furthermore, $\x$ is a fixed point of $\bH$ if and only if $(\x,\y)\in \C^n\times\C^n$ is a solution to the system:
\begin{eqnarray}\label{sysfixpH}
\bF(\x)-\y=0, \quad \bar{\bF}(\y)-\x=0.
\end{eqnarray}
Moreover
\begin{eqnarray}\label{Sinftynrmformc}
\|\cS\|_{\sigma}=\|f\|_{\sigma}=\max\{\frac{|f(\x)|}{\|\x\|^d}, \x\in\mathrm{fix}(\bH)\setminus\{\0\}\}.
\end{eqnarray}
\item Assume that $\cS\in\rS^d\C^n$ is nonsingular.  Then $\mathrm{fix}(\bF)$ and $\mathrm{fix}(\bH)$ have cardinalities $(d-1)^{n}$ and $(d-1)^{2n}$ respectively, counted with multiplicities.  The origin $\x=\0$ is a fixed point of $\bF$ and $\bH$ of multiplicity one.  Let $\x\in \mathrm{fix}(\bF)\setminus\{\0\}$ and $\y\in \mathrm{fix}(\bH)\setminus\{\0\}$.
Then $\phi\x\in \mathrm{fix}(\bF)\setminus\{\0\}$ and $\psi\y\in \mathrm{fix}(\bH)\setminus\{\0\}$ if and only if $\phi^{d-2}=1$ and $\psi^{(d-1)^2-1}=1$.
\end{enumerate}  
\end{theorem}
\begin{proof} (1-2) Let 
\begin{eqnarray*}
\begin{cases}
\alpha_{\R}=\sup\{\frac{|f(\z)|}{\|\z\|^d}, \z\in\mathrm{fix}_{\R}(\bF)\setminus\{\0\}\},\\
\alpha_{\C}=\sup\{\frac{|f(\z)|}{\|\z\|^d}, \z\in\mathrm{afix}(\bF)\setminus\{\0\}\}.
\end{cases}
\end{eqnarray*}
The characterization \eqref{polspecnorm} yields that $\alpha_{\F}\le \|f\|_{\sigma,\F}$.
Corollary \ref{maxeig} claims that  $\|\cS\|_{\sigma,\F}=|\lambda^\star|$, where $| \lambda^\star|$ is the maximum of all $|\lambda|$ satisfying (\ref{critpt}).  As $f\ne 0$ it follows that $\lambda^\star\in\R\setminus\{0\}$.  From the discussion before this theorem it follows that we can assume that $\lambda^*>0$ unless $d$ is even and $\F=\R$.

Assume that $\bu\in\rS(n,\F)$ satisfies $\bF(\bu)=\lambda^\star \bar\bu$.   (I.e. $\bu$ satisfies (\ref{critpt}) with $\lambda=\lambda^\star$).  Then there exists a positive $t$ such that  $|\lambda^\star| t^{d-1}=t$.  Let $\x=t\bu$.   Then $\bF(\x)=\frac{\lambda^\star}{|\lambda^\star|}\bar\x$.  First assume that $\lambda^\star>0$.  Then
$\x\in\mathrm{afix}(\bF)$.
  Furthermore, if $\F=\R$ then 
$\x\in\mathrm{fix}(\bF)\cap\R^n\subseteq \mathrm{fix}_{\R}(\bF)$.
Clearly, $\|\cS\|_{\sigma,\F}=|\lambda^\star|=\frac{|f(\x)|}{\|\x\|^d}$.  Hence $\|\cS\|_{\sigma,\F}=\alpha_{\F}$ in this case.

Second assume that $\F=\R$, $d$ is even and $\lambda^\star<0$.  Then $\bF(\x)=-\x$ and $\x\in\R^n$.  Define $\y=\omega_{d-2}\x$.  Then $\bF(\y)=\y$.
Clearly $\bar\omega_{d-2}\y=\x\in\R^n$.  Hence $\x\in \mathrm{fix}_{\R}(\bF)$.  As 
$\|\cS\|_{\sigma,\R}=|\lambda^\star|=\frac{|f(\x)|}{\|\x\|^d}$ we deduce that $\alpha_{\R}=\|\cS\|_{\sigma,\R}$.  This show the characterization \eqref{Sinftynrmformr}.

Assume that $\bF(\x)=\bar \x$. Then 
$$\bH(\x)=\overline{\bF}(\bF(\x))=\overline{\bF}(\bar \x)=\overline{\bF(\x)}=\x.$$
Hence $\mathrm{afix}(\bF)\subseteq \mathrm{fix}(\bH)$.   Similarly, $\x\in \mathrm{fix}(\bH)$ if and only if \eqref{sysfixpH} holds. (Set $\y=\bF(\x)$.)
Let $\beta$ be the right hand side of \eqref{Sinftynrmformc}.  By definition $\|f\|_{\sigma}\ge \beta$.  As $\mathrm{afix}(\bF)\subseteq \mathrm{fix}(\bH)$ it follows that $\alpha_{\C}\le \beta$.  Hence $\|f\|_{\sigma}=\beta$ and \eqref{Sinftynrmformc} holds.

\noindent
(3) Recall that $\mathrm{fix}(\bF)$ is the set of solutions of the system $\bF(\x)-\x=\0$.
As $d\ge 3$ the highest homogeneous part of this system is $\bF(\x)=\cS\times \otimes^{d-1}\x$.  As $\cS$ is nonsingular we deduce that $\bF(\x)=\0$ has the only solution $\x=\0$.  That is, the system $\bF(\x)-\x=\0$ does not have solutions at infinity.  Therefore the Bezout theorem yields that the number of solutions counting with multiplicities is $(d-1)^n=\prod_{i=1}^n \deg \bF_i$.  (See \cite{Fri77} for a proof using degree theory.)

Observe next that $\0\in \mathrm{fix}(\bF)$.  As the Jacobian $D(\bF(\x)-\x)$ is $-I$ at $\x=\0$ it follows that $\0$ is a fixed point of multiplicity one.  Assume that $\x\in\mathrm{fix}(\bF)\setminus\{\0\}$.  Then $\bF(\phi\x)=\phi^{d-1}\x=\phi^{d-2}(\phi \x)$.  Hence $\phi\x$ is a nonzero fixed point of $\bF$ if and only if $\phi^{d-2}=1$.

To show similar results for $\bH$ we first have to show that $\bH(\x)=\0$ has the only solution $\x=\0$.  As $\cS$ is nonsingular it follows that $\bar\cS$ is nonsingular.
Thus 
$$\bH(\x)=\bar{\bF}(\bF(\x))=\0\Rightarrow\bF(\x)=\0\Rightarrow \x=\0.$$
As $\deg \bH_i=(d-1)^2$ for $i\in[n]$ we deduce the similar results for $\bH$.
\end{proof}
Thus our approach to compute the spectral norm of $\cS$ is to compute the fixed points of $\bF$ and $\bH$ using the available software as Bertini \cite{BHSW06} for polynomial system of equations, and then use \eqref{Sinftynrmformr} or \eqref{Sinftynrmformc}.
Note that to compute the fixed points of $\bH$ we can use also the system \eqref{sysfixpH}.

 In \cite{HS14} the authors consider the dynamics of a special anti-holomorphic map of $\C$ of the form $z\mapsto \bar z^d +c$.  They also note that the dynamics of the ``squared" map is given by the holomorphic map $z\mapsto (z^d +\bar c)^d +c$.
Thus the dynamics of the maps $\x\mapsto \overline{\bF(\x)}$ and its square - $\bH$ are generalizations of the dynamics studied in \cite{HS14}.

Suppose that $\F=\C$.  Assume that $\x\in\rS(n,\C)$ and $\lambda\in\C$ are an eigenvector and the corresponding eigenvalue of $\cS\in\rS^d\C^n$, i.e., \eqref{defeigvl} holds.  Let $\zeta\in\C,|\zeta|=1$.
Then $\zeta\x$ is an eigenvector of $\cS$ with the corresponding eigenvalue $\zeta^{d-2}\lambda$.  Assume that $\lambda\ne 0$.  For $d>2$ we can choose $\zeta$ of modulus $1$ such that $\zeta^{d-2}\lambda=|\lambda|>0$.  
Furthermore,  the number of such choices of $\zeta$ is $d-2$.   In this context it is natural to consider the eigenspace $\mathrm{span}(\x)$, to which correspond a unique eigenvalue
$\lambda\ge 0$.

It is shown in \cite{CS} that the number of different 
eigenspaces of generic $\cS\in\rS^d\C^n$ is
\begin{eqnarray*}\label{CSnumbereigv}
c(2,n) =n, \quad c(d,n)=\frac{(d-1)^n -1}{d-2} \textrm{ for } d\ge 3.
\end{eqnarray*}
For $d\ge 3$ and a nonsingular $\cS$ this result follows from part (3) of Theorem \ref{theofoundthm}, where we count the number of nonzero fixed points of $\bF$.
That is, each $\cS\in\rS^d\C^n\setminus V(d,n)$ has the above number of eigenspaces $\mathrm{span}(\x), \x\in\rS(n,\C)$.  The obvious question is: what is the maximal number
of eigenspaces $\mathrm{span}(\x)$ corresponding to $\x\in\rS(n,\R)$ for  $\cS\in\rS^d\R^n\setminus V(d,n)$.  Since $\cS\times \otimes^d\x$ has at least two critical points on
$\rS(n,\R)$ for $\cS\ne 0$, corresponding to the maximum and minimum values, it follows that $\cS\ne 0$ has at least one real eigenspace. 
In \cite{ABC13} the authors study the average number of critical points of a random homogeneous function $f(\x)$ of degree $d$, where its coefficients are independent Gaussian random variables.

Assume that $d=2$.  Then $\rP(2,n,\F)$ is the space of quadratic forms in $n$ variables on $\F$, which correspond to the space of symmetric matrices $\rS^2\F^n$.
That is $f(\x)=\x\trans S\x$, where $S\in \F^{n\times n}$ is symmetric.
 For $\F=\R$ the critical points of $f(\x)$ correspond to the eigenvalues of $S$.  For $\F=\C$ recall Schur's theorem:  
There exists a unitary matrix $U\in \C^{n\times n}$ such that 
$U\trans SU=\diag(a_1,\ldots,a_n), a_1\ge \cdots\ge a_n\ge 0 $, where $\diag(a_1,\ldots,a_n)\in\rS^2\C^n$ is a diagonal matrix with the diagonal entries $a_1,\ldots,a_n$. As $\bar U$ is unitary it follows that $a_i=\sigma_i(S), i\in [n]$ are the singular values of $S$.    
Let $U=[\bu_1,\cdots,\bu_n] $.  Then $SU=\bar U \diag(a_1,\ldots,a_n)$ which is equivalent to $S\bu_i=a_i \bar\bu_i, i\in[n]$, which is a special case of (\ref{critpt}).  

We now give an estimate of the number of different positive anti-eigenvalues for a nonsingular $\cS\in\rS^d\C^n$.

\begin{theorem}\label{estnumbanteig}  Assume that $\cS\in\rS^d\C^n$ is nonsingular.  Then the number of positive anti-eigenvalues with corresponding anti-eigenspaces
	is finite.  This number $\mu(\cS)$, counting with multiplicities,  satisfies the inequalities 
	\begin{eqnarray*}\label{estnumbanteig1}
	\frac{(d-1)^n-1}{d}\le \mu(\cS)\le \frac{(d-1)^{2n} -1}{(d-1)^2-1}=\sum_{k=0}^{n-1} (d-1)^{2k}.
	\end{eqnarray*}
\end{theorem}
\begin{proof} Assume that $\cS\in\rS^d\C^n$ is nonsingular.  First suppose  that $d=2$. Schur's theorem implies that the number of different positive anti-eigenvalues of a complex symmetric 
	matrix, which are the singular values of $\cS$, is at most $n$.  Hence our theorem holds. 
	
	Second suppose that $d>2$.
	Assume that $\x\in\rS(n,\C)$ is an anti-eigenvector with corresponding anti-eigenvalue $\lambda> 0$.  As in the proof of Theorem \ref{theofoundthm} we can assume that $\y\in\mathrm{afix}(\bF)\setminus\{\0\}$. Recall that $\mathrm{afix}(\bF)\setminus\{\0\}\subset\mathrm{fix}(\bH)$.  Theorem \ref{theofoundthm} yields that
$\mathrm{fix}(\bH)\setminus\{\0\}$ has cardinality $(d-1)^{2n}-1$.  The subspace spanned by $\y\in \mathrm{fix}(\bH)\setminus\{\0\}$ contains $(d-1)^2-1$ fixed points corresponding to $\psi \y$, where $\psi^{(d-1)^2-1}=1$.  This shows the upper bound in \eqref{estnumbanteig1}.

We now show the lower bound using the degree theory as in \cite{Fri77}.  
Let $\mu=\min\{\|\bF\|, \|\x\|=1\}$.  As $\bF(\x)=\0\Rightarrow \x=\0$ it follows that $\mu>0$.
Let $\bG_t(\x)=\bF(\x)-t\bar(\x)$ for $t\in[0,1]$. Then $\|\bG_t(\x)\|\ge \mu\|\x\|^d -t\|\x\|$.In particular, $\lim_{\|\x\|\to\infty}\|\bG_t(\x)\|=\infty$. Hence
 $\bG_t: \C^n\to\C^n$ is a proper map.   Let $\C^n\cup\{\infty\}$ be one point compactification of $\C^n$.  So  $\C^n\cup\{\infty\}$ is homeomorphic to the $2n$ dimensional sphere $\rS^{2n}$.  Thus $\bG_t$ extends to a continuous map $\hat\bG_t:\rS^{2n}\to\rS^{2n}$.  Let $\deg \hat\bG_t$ be the topological degree of $\hat\bG_t$.  The above arguments so that this topological degree is constant for $t\in[0,1]$.  Hence $\deg \hat\bG_1=\deg \hat\bG_0=\deg \hat\bF$.  The topological degree $\hat\bF$ is just the covering degree of the proper complex polynomial map $\bF$, which is $(d-1)^n$.  Hence the number of the antifixed points of $\bF$ is at least $\deg \hat\bG=(d-1)^n$.
 As $\0$ is a simple fixed point of $\bH$, $\0$ is a simple zero of $ \bG$.  Therefore $|\mathrm{afix}(\bF)\setminus\{\0\}|\ge (d-1)^n-1$.  Recall that if $\x\in \mathrm{afix}(\bF)\setminus\{\0\}$ then $\zeta\x\in\mathrm{afix}(\bF)\setminus\{\0\}$ for $\zeta^d=1$.  This establishes the lower bound in \eqref{estnumbanteig1}.  
\end{proof} 

Consider the following example: $f(\x)=\sum_{i=1}^n x_i^d$.  Then $\bF(\x)=(x_1^{d-1},\ldots,x_n^{d-1})\trans$.  It is straightforward to show that $|\mathrm{afix}(\bF)|=(d+1)^n$.  Furthermore the number of positive eigenvalues of the corresponding $\cS$ is $\frac{(d+1)^{n-1}-1}{d}$.

In what follows we will need the following observation:
\begin{lemma}\label{FHest} Assume that $\cS\in\rS^d\F^n\setminus\{0\}$.  Let $\bF(\z)=\cS\times(\otimes^{d-1}\z)$ and $\mathbf{H}=\bar \bF\circ \bF$.
	Then
	\begin{eqnarray*}\label{FHest1}
	\|\bF(\z)\|\le \|\cS\|_{\sigma,\F}\|\z\|^{d-1}, \quad \|\bH(\z)\|\le \|\cS\|_{\sigma,\F}^d \|\z\|^{(d-1)^2}, \quad \z\in\F^n.
	\end{eqnarray*}
	For $\z\in\rS(n,\F)$ satisfying $|\cS\times \otimes^d\z|=\|\cS\|_{\sigma,\F}$ equality holds in the above inequalities.  Suppose furthermore that $d>2$ and $\x\in \F^n\setminus\{\0\}$ is a fixed point of $\bH$.  Then
	\begin{eqnarray*}\label{FHest2}
	\|\x\|^{-(d-2)}\le \|\cS\|_{\sigma,\F},
	\end{eqnarray*}	
and this inequality is sharp.
\end{lemma}
\begin{proof}  Since $\bF$ and $\mathbf{H}$ are homogeneous maps of degree $d-1$ and $(d-1)^2$ respectively, it is enough to prove the first two inequalities of our lemma for $\z\in\rS(n,\F)$.
	Assume that $\z\in\rS(n,\F)$.  
	Let $\bw=\cS\times \otimes^{d-1}\z$.  First assume that $\bw=\0$.  Then $\bF(\z)=\bH(\z)=\0$ and the first two inequalities of our lemma trivially hold.
	Second assume that $\bw\ne \0$.  
	Let $\bu=\frac{1}{\|\bw\|}\overline {\bw}$.  Hence 
	\[\|\bF(\z)\|=|\cS\times (\bu\otimes (\otimes^{d-1}\z))|\le \|\cS\|_{\sigma,\F}.\]
	This establishes the first inequality of our lemma.  Clearly, $\|\bar\cS\|_{\sigma,\F}=\|\cS\|_{\sigma,\F}$.  Hence
	\[\|\bH(\z)\|=\|\bar\bF(\bF(\z))\|\le \|\cS\|_{\sigma,\F}(\|\bF(\z)\|)^{d-1}\le  \|\cS\|_{\sigma,\F} ( \|\cS\|_{\sigma,\F})^{d-1}= \|\cS\|_{\sigma,\F}^d.\]
	This establishes the second inequality of our lemma.
	
	Suppose that  $|\cS\times \otimes^d\z|=\|\cS\|_{\sigma,\F}$ for $\z\in\rS(n,\F)$.
	First assume that $\F=\C$.
	Hence there exists $\zeta\in\C,|\zeta|=1$ such that $\x=\zeta\z$ satisfies (\ref{critpt}) with $\lambda=\|\cS\|_{\sigma}$.  Clearly $\|\bF(\z)\|=\|\bF(\x)\|=\lambda=\|\cS\|_{\sigma}$.
	Moreover 
	\[\bH(\x)=\bar \bF(\lambda\bar\x)=\lambda^{d-1}\bar \bF(\bar\x)=\lambda^d\x=
	\|\cS\|^{d}_{\sigma}\x.\]
	Hence $\|\bH(\z)\|=\|\bH(\x)\|=\|\cS\|^{d}_{\sigma}$.  
	
	Second assume that $\F=\R$.  Then $\z\in\rS(n,\R)$ is a critical point of $\cS\times \otimes^d\x$ on $\rS(n,\R)$.  Corollary \ref{maxeig} yields that $\cS\times \otimes^{d-1}\z=\pm \|\cS\|_{\sigma,\R}\z$.  Hence $\|\bF(\z)\|=\|\cS\|_{\sigma,\R}$ and $\|\bH(\z)\|=\|\cS\|_{\sigma,\R}^d$.
	
	Assume finally that $\bH(\x)=\x, \x\ne \F^n\setminus\{\0\}$.  The second inequality of our lemma yields $\|\x\|= \|\bH(\x)\|\le  \|\cS\|^d_{\sigma,\F}\|\x\|^{(d-1)^2}$.  Hence
	$ \|\cS\|^d_{\sigma,\F}\ge \|\x\|^{-(d-1)^2+1}=\|\x\|^{-d(d-2)}$ which yields the third inequality of our lemma. If $\x$ corresponds to the critical vector $\z\in\rS(n,\F)$ with the eigenvalue $\lambda$ satisfying $|\lambda|=\|\cS\|_{\sigma,\F}$ then $\|\x\|
^{-(d-2)}= \|\cS\|_{\sigma,\F}$.
\end{proof}
\section{Polynomial-time computability of the spectral norm of $\cS\in\rS^d\F^n$ for a fixed $n$}\label{sec:dnqudit}
In this section we assume that $d\ge 3$ and $\cS\ne 0$.  (For $d=2$, (matrices), the spectral norm is the maximal singular value of the matrix, which is polynomially computable.)  

{Furthermore we are going to use the results of Appendix 2. 
\begin{definition}\label{dersnsing}A symmetric tensor $\cS\in \rS^d\C^n$, and the corresponding polynomial $f(\x)=f_{\cS}(\x)=\cS\times\otimes^d\x$,  are called strongly nonsingular if the following conditions hold: First, $d\ge 3$ and $\cS$ is nonsingular, i.e. the hypersurface $f=0$ is smooth in $\mathbb{P}\C^n$.   Second, the $x_1$ coordinates of $(d-1)^{2n}$  solutions $(\x,\y)$ of the system \eqref{sysfixpH} are distinct.
\end{definition}
\begin{lemma}\label{varsnons}  The set of  $f\in \rP(d,n,\C)$ which are not strongly nonsingular, denoted as $\rP_{ss}(d,n,\C)$, has the following structure: Identify $\rP(d,n,\C)$ with $\C^{J(d,n)}$ and $\rP_{ss}(d,n,\C)$ with $\rV_{ss}(d,n)$.  Then  $\rV_{ss}(d,n)$ is a disjoint union of hyperdeterminant variety $\rH(d,n)$ and the set $\rV'_{ss}(d,n)\subset\C^{J(d,n)}$ which  is characterized as follows.  There exists a bihomogeneous  polynomial $p(\bu, \bv), (\bu,\bv)\in C^{J(d,n)}\times C^{J(d,n)}$ of total degree $2((d-1)^{2n}-1)$ and of degree $(d-1)^{2n}-1$ in $\bu$ and $\bv$ such that $\bu\in \rV_{ss}'(d,n)$ if and only if $\bu\not\in \rH(d,n)$ and $p(\bu,\bar\bu)=0$. 
\end{lemma}
}
\begin{proof} Let $f,g\in\rP(d,n,\C)$.  Define $\bF=\frac{1}{d}\nabla f, \bG=\frac{1}{d}\nabla g$.  Consider a generalization of the system \eqref{sysfixpH}:
\begin{eqnarray}\label{FGsys} \bF(\x)-\y=\0, \quad \bG(\y)-\x=\0.
\end{eqnarray}
As $d\ge 3$, the homogeneous part of the above system, i.e., the system  \eqref{trivsolcond}, $\bF(\x)=\0,\bG(\y)=\0$ has a unique solution $\x=\y=\0$ if and only if $f=0$ and $g=0$f are smooth hypersurfaces.  That is, $f,g\not\in \rH(d,n)$.  In this case the system \eqref{FGsys} has $D=(d-1)^{2n}$ isolated solutions, counting with multiplicities, and no solution at infinity.  (See Appendix 2).  
Thus if we find the reduced Gr$\mathrm{\ddot{o}}$bner  basis with respect to the order
\begin{eqnarray}\label{lexordxy}
x_1\prec \cdots\prec x_n\prec y_1\prec \cdots \prec y_n
\end{eqnarray}
then the last polynomial is a monic polynomial $p_1(x_1)$ of degree $(d-1)^{2n}$.
We now show an example where $p_1(x_1)$ has $(d-1)^{2n}$ distinct roots.

Let $f=g=h$, where $h=\sum_{i=1}^n x_i^d$.  Then the system \eqref{FGsys} is the system \eqref{sysfixpH}.  It splits to $n$ systems in $(x_k,y_k)$ for $k\in[n]$:
\begin{eqnarray*} x_k^{d-1}=y_k, \quad y_k^{d-1}=x_k, \quad k\in[n].
\end{eqnarray*}
Note that this system has $D$ distinct solutions.  However $x_1$ in these solutions has only $(d-1)^2$ distinct values: $0$, and $(d-1)^{2}-1$ roots of unity.  We now show how perturb $h$ so that the $x_1$-coordinates of the solutions of  \eqref{FGsys} for $f=g=h$ are distinct.  Let $\ba=(a_2,\ldots,a_n)\trans\in\C^{n-1}$  and denote
$h_{\ba}(\x)=(x_1+\sum_{j=2}a_jx_j)^d+\sum_{i=2}^n x_i^d$.  It is straightforward to check that $h_{\ba}$ is nonsingular.  Assume that $\ba$ is close to $\0$.  Then  the $D$ solutions of \eqref{FGsys} corresponding to $h_{\ba}$ are close the $D$ distinct solutions of   \eqref{FGsys} corresponding to $h_{\0}$.  In particular each solution of  of \eqref{FGsys} corresponding to $h_{\ba}$ is analytic in $\ba$ in the neighborhood of $\ba=\0$.  It is left to show that we can choose $\ba$ close to $\0$ such that $x_1(\ba)$ are all distinct.  Let $(\x_l(\ba),\y_l(\ba))=(x_{1,l}(\ba),\ldots, x_{n.l}(\ba),y_{1,l}(\ba),\ldots,y_{n,l}(\ba))\trans$ be the $l$-analytic solution of the system \eqref{FGsys} corresponding to $h_{\ba}$ for $l\in[D]$.  Clearly if $x_{1,p}(\0)\ne x_{1,q}(\0)$ then $x_{1,p}(\ba)\ne x_{1,q}(\ba)$ for a small $\ba$.   Thus we need to study the case where $x_{1,p}(\0)=x_{1,q}(\0)=\zeta$.  That is, $\zeta$ is a root of $z^{(d-1)^2} -z=0$.  

The two equations of the system  \eqref{FGsys} corresponding $F_1(\x)-y_1=0$ and  $G_1(\y)-x_1=0$ is 
\begin{eqnarray*}
(x_1+\sum_{j=2}a_jx_j)^{d-1}-y_1=0, \quad (y_1+\sum_{j=2}a_jy_j)^{d-1}-x_1=0.
\end{eqnarray*}
To find the $\nabla x_{1,p}(\0)$ we can assume that $x_{i,l}(\ba)=\x_{i,l}(\0)$ and $y_{i,l}(\ba)=y_{i,l}(\0)$ for $i\ge 2$.  Observe that $y_{i,l}(\0)=x_{i,l}^{d-1}(\0)$ for $i\in[n]$ and $l\in[D]$.
First assume that $\x_{1,p}(\0)=\zeta\ne 0$.  Hence $\zeta^{(d-1)^2-1}=1$.  Thus
\begin{eqnarray*}
&&x_{1,p}(\ba)=(y_{1,p}(\ba)+\sum_{j=2}^n a_jy_{p,j}(\ba))^{d-1}=y^{d-1}_{1,p}(\ba) +(d-1)\zeta^{d-2}\sum_{j=2}^n a_jy_{j,p}(\0)+ \cO(\|\ba\|^2)=\\
&&(x_{1,p}(\ba)+\sum_{i=2}^n a_i x_{j,p}(\ba))^{(d-1)^2}+(d-1)\zeta^{d-2}\sum_{j=2}^n a_jx_{j,p}^{d-1}(\0)+ \cO(\|\ba\|^2)=\\
&&x_{1,p}^{(d-1)^2}(\ba)+(d-1)^2\zeta^{(d-1)^2-1}\sum_{j=2}^n a_jx_{j,p}(\0)
+(d-1)\zeta^{d-2}\sum_{j=2}^n a_jx_{j,p}^{d-1}(\0)+ \cO(\|\ba\|^2).
\end{eqnarray*}
As $\zeta^{(d-1)^2-1}=1$ we obtain the equation
\begin{eqnarray*}
x_{1,p}(\ba)-x_{1,p}(\ba)^{(d-1)^2}=(d-1)\sum_{j=2}^n ((d-1)x_{j,p}(\0)+\zeta^{d-2}x_{j,p}^{d-1}(\0))a_j + \cO(\|\ba\|^2).
\end{eqnarray*}
Hence
\begin{eqnarray*}
\nabla x_{1,p}(\0)=-\frac{d-1}{d(d-2)}\sum_{j=2}^n ((d-1)x_{j,p}(\0)+\zeta^{d-2}x_{j,p}^{d-1}(\0))a_j.
\end{eqnarray*}
Recall that $x_{j,p}(\0)^{(d-1)^2}=x_{j,p}(\0)$ for $j\ge 2$ and $p\in[D]$.  Hence, for generic real $a_2,\ldots,a_n$ the values of $\nabla x_{1,p}(\0)$ are disrinct for all solution $\x_{p}(\ba)$ such that $x_{1,p}(\0)=\zeta\ne 0$.

Assume $\zeta=0$.  The above arguments yield that $\nabla x_{1,p}(\ba)=\0$.  Similarly, $\nabla y_{1,p}(\ba)=\0$.  Hence the power series of $x_{1,p}(\ba)$ and $x_{1,p}(\ba)$ start from at least a quadratic polynomial.  As $x_{1,p}(\ba)=(y_{1,p}(\ba)+\sum_{j=1}^n a_jy_{j,p}(\ba))^{d-1}$ we deduce that the power series of $x_{1,p}(\ba)$ start with a homogeneous polynomial of degree $d-1$ of the form
$(\sum_{j=2}^n a_j y_{j,p}(\0))^{d-1}$.  Recall that $y_{j,p}^{(d-1)^2}=y_{j,p}(\0)$ for $j\ge 2$ and $p\in[D]$.  Hence for generic real $a_2,\ldots,a_n$ all polynomials $(\sum_{j=2}^n a_j y_{j,p}(\0))^{d-1}$ are different.  Hence for a small real generic values of $a_2,\ldots,a_n$ we have that $p_1(x_1)$ have $D$ distinct roots.  (Note that $x_1=0$ is always a root of $p_1(x_1)$.)

Assume that $f,g\in\rP(n,d,\C)$ are represented by $\bu,\bv\in\C^{J(d,n)}$.  Suppose that $\bu,\bv\not\in \rH(d,n)$.  Then $p_1(x_1)$ is a polynomial of degree $D$.  Its coefficients are rational functions in $\bu,\bv$.  By multiplying $p_1(x_1)$ by a corresponding polynomial in $\bu,\bv$ we obtain a polynomial $P_1(x_1)$ of degree $D$ whose coefficients are polynomials in $\bu,\bv$ each one of degree $D$.
$P_1(x_1)$ will not have $D$ distinct roots if and only if the discriminant of $P_1(x_1)$ is zero.  This discriminant is a polynomial $p(\bu,\bv)$ of degree $2(D-1)$.
It is not hard to see that $p(\bu,\bv)$ is a bihomogeneous polynomial in $(\bu,\bv)$ of degree $D-1$ in $\bu$ and $\bv$ respectively.  

Observe next that the system \eqref{sysfixpH} corresponds to a point $(\bu,\bar\bu)$.
For a real $\bu$ the system \eqref{FGsys} corresponds to $f=g$.  We showed that for a generic choice 
of $\bu$ $p(\bu,\bu)\ne 0$. Thus $V_{ss}(d,n)=\rH(d,n)\cup \rV'_{ss}(d,n)$, where 
$\rV'_{ss}(d,n)=\{\bu\in\C^{J(d,n)}, \bu\notin \rH(d,n), \;p(\bu,\bar\bu)=0\}$.
\end{proof}  

Let $\bi=\sqrt{-1}\in\C$, i.e., $\bi^2=-1$.  Denote by $\Z[\bi]=\Z+\bi\Z\subset\C$ the integral domain of Gaussian integers, and by $\Q[\bi]$ the field of Gaussian rationals.  Let $\Z^n\subset \R^n$ and $\Z[\bi]^n\subset \C^n$ be the $\Z$ and $\Z[\bi]$ modules of vectors with integer and Gaussian integer coordinates respectively.  
We now give an upper bound on the complexity of finding the spectral norm $\|\cS\|_{\sigma}$, assuming first that $\cS$ is strongly nonsingular and the entries of $\cS$ are Gaussian rationals, i.e., $\cS\in\rS^d \Q[\bi]^n$. (Note that the assumption that $\cS$ is strongly nonsingular yields that $\cS\ne 0$.)  Equivalently, we can assume that $\cS=\frac{1}{N}\cT$, where $\cT$ is a symmetric tensor with Gaussian integers entries $\cT\in\rS^d\Z[\bi]^n$ and $N\in\N$.  Thus it is enough to estimate the spectral norm of $\cT$.  We identify $\cT$ with $f(\x)=\cT\times(\otimes^d \x)$.  We assume that each coefficient $f_{\bj}$ in \eqref{defpolx1} is $a_{\bj}+\bi b_{\bj}$, where $a_{\bj},b_{\bj}\in\Z$ and $|a_{\bj}|,|b_{\bj}|\le 2^{\tau}$ for each $\bj\in J(d,n)$ for a given integer $M\in\N$.  

Next we compute $\bF=(F_1,\ldots,F_n)=\frac{1}{d}\nabla f$.   As $\bF=\cT\times \otimes^{d-1}\x$, it follows that the coefficient of each monomial in the coordinates of $\bF$ is a Gaussian integer.  Hence the bit length of an integer coefficient is $\tau$. 
As $\cT$ is strictly  nonsingular we deduce that the system \eqref{sysfixpH} has exactly $D=(d-1)^{2n}$ simple solutions.  Furthermore, after finding the reduced Gr$\mathrm{\ddot{o}}$bner basis with respect to the lexicographical order \eqref{lexordxy} we have the conditions of Lemma \ref{shapelem}, the Shape Lemma, in Appendix 2.  That is, the reduced Gr$\mathrm{\ddot{o}}$bner basis is of the form
\begin{eqnarray*}
p_1(x_1), x_2-p_2(x_1), \ldots,x_n-p_n(x_1), y_1-p_{n+1}(x_1),\ldots,y_n-p_{2n}(x_1).
\end{eqnarray*}
The degree of $p_1(x_1)$ is $D$, and $p_1(x_1)$ has simple zeros.  The degree of each $p_i(x_1)$ is less than $D$ for $i>1$. 
The solutions of the system  \eqref{sysfixpH} are parametrized by the roots $x_1$ of $p_1(x_1)$.  They are of the form $(\x,\y)$, where $\x=(x_1,p_2(x_1),\ldots,p_n(x_1)\trans$ and $\y=(p_{n+1}(x_1),\ldots,p_{2n}(x_1))\trans$.  Recall that the fixed points of $\bH$ are the $\x$ part of the solutions $(\x,\y)$ of the solutions of \eqref{sysfixpH}.  Let
\begin{eqnarray}\label{solfixH}
X=\{(x_1,p_2(x_1),\ldots,p_n(x_1))\trans\in \C^n, \;p_1(x_1)=0\}.
\end{eqnarray}
Thus $X$ is a parametrization of all fixed points of $\bH$ for strongly nonsingular $\cT\in\rS^d\C^n$.  In particular, $|X|=D$.  
 Lemma \ref{eransol} of Appendix 2 gives the bit complexity of computing the coordinate of each solution $(\x,\y)$ of \eqref{sysfixpH} with precision $2^{-\ell}$, for a given $\ell\in\N$.

 Recall that for a nonzero fixed point $\x$ of $\bH$ corresponding to $\cS$ the third inequality of Lemma \ref{FHest} holds.  As $\|\cT\|_{\sigma}\le \|\cT\|$ we obtain that a nonzero fixed point of $\bH$ corresponding to $\cT$ satisfies the inequality $\|\x\|\ge \|\cT\|^{-\frac{1}{d-2}}$.  Hence  Lemma \ref{FHest} yields
\begin{equation}\label{charcTspecnrm}
\|\cT\|_{\sigma}=\max\{\|\x\|^{-(d-2)}, \; \x\in X, \; \|\x\|\ge \|\cT\|^{-\frac{1}{d-2}}\}, \quad f(\by)=\cT\times(\otimes^d \x).
\end{equation}

\begin{theorem}\label{cesnsingcten}  Let $d\ge 3$ be an integer.  Assume that $\cT\in\rS^d\Z[\bi]^n$, and each coordinate of $\cT$ is bounded above by $2^{\tau}$ for some $\tau\in\N$.  If
$\cT$ is strongly nonsingular then for a given $e\in\N$ we can compute rational $L(\cT)$ satisfying 
\begin{eqnarray*}\label{Lctineq}
 |\|\cT\|_{\sigma}-L(\cT)|\le 2^{-e} \|\cT\|_{\sigma}.
\end{eqnarray*}
The bit complexity of computing $L(\cT)$ is $\tilde O\big((\tau+e) d^{8n}\big)$ 
\end{theorem} 
\begin{proof}  Recall that $\x\in X\setminus\{\0\}$ satisfies the inequality $\|\x\|\ge \|\cT\|^{-\frac{1}{d-2}}$.  As $\cT\in \rS^d \Z[\bi]^n\setminus\{0\}$ it follows that 
\begin{eqnarray*}
1\le\|\cT\|\le \sqrt{n+d-1\choose d}2^{\tau}\le (d+1)^{(n-1)/2} 2^{\tau}.
\end{eqnarray*}
We compute the coordinates of  $\x\in X\setminus\{\0\}$ with precision $2^{-(e+k)}$, where $k\in\N$ is specified below. This will give an approximation $\hat \x(\x)=(\hat x_1,\ldots,\hat x_n)\trans$ of $\x$.  Observe that 
$\|\x-\hat\x(\x)\|\le \sqrt{n}2^{-(e+k)}$.  We assume that
\begin{eqnarray*}
\sqrt{n}2^{-k}\le 2^{-d}(d+1)^{-(n-1)/2(d-2)}2^{-\tau/(d-2)}.
\end{eqnarray*}
To satisfy the above inequality we choose
\begin{eqnarray}\label{defkLTest}
k=\lceil d+(1/2)\log_2 n+\frac{n-1}{2(d-2)}\log_2(d+1)+\tau/(d-2)\rceil.
\end{eqnarray}
Hence for any nonzero fixed point one has the inequality
\begin{eqnarray*}
\|\x-\hat \x(\x)\|\le \sqrt{n}2^{-(e+k)}\le 2^{-(e+d)}\|\cT\|^{-\frac{1}{d-2}}\le 2^{-(e+d)}\|\x\|.
\end{eqnarray*}
In particular $(1-2^{-(e+d)})\|\x\|\le \|\hat\x(\x)\|\le (1-2^{-(e+d)})\|\x\|$.
Let 
\begin{eqnarray*}
L(\cT)=\max\{\|\hat\x(\x)\|^{-(d-2)}, \x\in X\setminus\{\0\}\}.   
\end{eqnarray*}
We claim that the inequality \eqref{Lctineq} holds.  For a nonzero fixed point of $\bH$ we estimate 
$|\|\x\|^{-(d-2)}-\|\hat\x(\x)\|^{-(d-2)}|$.  First observe that 
\begin{eqnarray*}
&&|\|\x\|^{d-2}-\|\hat\x(\x)\|^{d-2}|\le |\|\x\|-\|\hat\x(\x)\||(d-2)\max(\|\x\|^{d-3}, \|\hat\x(\x)\|^{d-3})\le \\
&&2^{-(e+d)}(d-2)(1+2^{-(e+d)})^{d-2}\|\x\|^{d-2}.
\end{eqnarray*}
Hence
\begin{eqnarray*}
&&|\|\x\|^{-(d-2)}-\|\hat\x(\x)\|^{-(d-2)}|=|\|\x\|^{d-2}-\|\hat\x(\x)\|^{d-2}|\|\x\|^{-(d-2)}\|\hat\x(\x)\|^{-(d-2)}\le\\ 
&&2^{-(e+d)}(d-2)(1+2^{-(e+d)})^{d-2}(1-2^{-(e+d)})^{-(d-2)}\|\x\|^{-(d-2)}.
\end{eqnarray*}
As $e\ge 1$and $d\ge 3$ it follows that $(1+2^{-(e+d)})/(1-2^{-(e+d)})\le 17/15$.  It is straightforward to show that $(d-2)(1+2^{-(e+d)})^{d-2}(1-2^{-(e+d)})^{-(d-2)}\le 2^d$ for an integer $d\ge 3$.  Hence$|\|\x\|^{-(d-2)}-\|\hat\x(\x)\|^{-(d-2)}|\le 2^{-e}\|\x\|^{-(d-2)}$.

First choose a fixed point $\bx$ of satisfying $\|\x\|^{-(d-2)}=\|\cT\|_{\sigma}$.   Therefore $L(\cT)\ge (1-2^{-e})\|\cT\|_{\sigma}$.  Assume that $L(\cT)=\|\hat\x(\x)\|^{-(d-2)}$ for some $\x\in X\setminus\{\0\}$.  Then
$L(\cT)\le (1+2^{-e})\|\x\|^{-(d-2)}\le (1+2^{-e})\|\cT\|_{\sigma}$.  Thus $ |\|\cT\|_{\sigma}-L(\cT)|\le 2^{-e} \|\cT\|_{\sigma}$.

It is left to show that the bit complexity of computing $L(\cT)$ is $\tilde\cO\big((\tau+e)d^{8n}\big)$.   This follows from the proof of Lemma \ref{eransol} in Appendix 2.
First note that in Lemma \ref{eransol} $m=2n$.  Next observe that the value of $\ell$ in 
 Lemma \ref{eransol} is $e+k$, where $k$ is given by \eqref{defkLTest}.
\end{proof} 

We now present the main result of this section:
\begin{theorem}\label{cessingcten}  Let $d\ge 3$ be an integer.  Assume that $\cJ\in \rS^d\Z[\bi]^n$ is a strongly nonsingular tensor that satisfies $\|\cJ\|\le 2^{c}, c\in\N$.   Suppose that $\cT\in\rS^d\Z[\bi]^n$, and each coordinate of $\cT$ is bounded above by $2^{\tau},\tau\in\N$.  
For a given  $b,e\in\N$ we can compute  $L(\cT)$ satisfying 
\begin{eqnarray*}\label{Lctineqm}
 |\|\cT\|_{\sigma}-L(\cT)|\le 2^{-e} \|\cT\|, 
\end{eqnarray*}
with probability greater than $1-(d-1)^{-2nb}$.
The bit complexity of computing $L(\cT)$ is $\tilde O\big((\tau+c+2nb\log_2 (d-1)+e) d^{8n}\big)$.
\end{theorem} 
\begin{proof}
Clearly, it is enough to assume that $\cT\ne 0$.
Lemma \ref{varsnons} yields that the set of strongly singular symmetric tensors is the zero set of $\Delta(\bu)p(\bu,\bar\bu)$, where $\bu\in \C^{J(d,n)}$ represents the polynomial $f(\x)=\cS\times \otimes^d\x$, for $\cS\in\rS^d\C^n$.  Here $\Delta(\bu)$ is a polynomial of degree $n(d-1)^{n-1}$, and $\Delta(\bu)=0$ is the hyperdeterminant variety.
The polynomial $p(\bu,\bar\bu)$ is a polynomial of degree  $2((d-1)^{2n}-1)$.
Let $g(\x)=\cJ\times\otimes^d\x$.   Denote by $\bg\in\C^{J(d,n)}$ the vector corresponding to $g(\x)$.  We now consider the affine line of polynomials $f_t(\x)=t f(\x)+g(\x)$ for $t\in\R$.  The value of the polynomial $\Delta p$ on $t\ff+\bg$ is $q(t)=\Delta(t\ff+\bg)p(t\ff+\bg,t\bar\ff+\bar\bg)$.  As $\cJ$  is strongly nonsingular $q(0)\ne 0$.  Hence $q(t)$ is a nonzero polynomial of degree at most $n(d-1)^{n-1}+2((d-1)^{2n}-1)< 3(d-1)^{2n}$.   

Let $A=\{M,M+2,\ldots, M+N-1\}\subset \Z$.  We assume that $M=2^{c+e+2}$ and $N=3(d-1)^{2n(b+1)}$.
The cardinality of $A$ is $N$.  
For each $a\in A$, let us consider the tensor $\cT(a)=a\cT+\cJ$.   Note $q(t)$ vanishes at most at $n(d-1)^{n-1}+2((d-1)^{2n}-1)$ points of $A$.  Choose a random $a\in A$ from the uniform distribution on A.  Then with probability greater than $1-(d-1)^{-2nb}$ the tensor $\cT(a)$ is strongly nonsingular.

Let $L(\cT(a))$ be the approximation given by Theorem \ref{cesnsingcten}, where we replace $e$ by $e+2$.   We now choose $L(\cT)=\frac{1}{a} L(\cT(a))$.  We claim that
$ |\|\cT\|_{\sigma}-L(\cT)|\le 2^{-e} \|\cT\|$.  Indeed, the inequality of Theorem \ref{cesnsingcten} yields
\[|\|\cT+\frac{1}{a}\cJ\|_{\sigma}- L(\cT)|\le \frac{2^{-e}}{4} \|\cT+\frac{1}{a}\cJ\|_{\sigma}.\]
As $\cT\in\rS^d\Z[\bi]\setminus\{0\}$ it follows that $\|\cT\|\ge 1$.  Since $a\ge 2^{c+e+2}$ and $\|\cJ\|\le 2^{c}$ we deduce
\begin{eqnarray*}
&&|\|\cT\|_{\sigma}-L(\cT)|\le |\|\cT+\frac{1}{a}\cJ\|_{\sigma}-L(\cT)\|+|\|\cT\|_{\sigma}-\|\cT+\frac{1}{a}\cJ\|_{\sigma}\|\le\\
&&2^{-(e+2)}\|\cT+\frac{1}{a}\cJ\|_{\sigma}+\frac{1}{a}\|\cJ\|_{\sigma}\le 
2^{-(e+2)}\|\cT\|_{\sigma}+(2^{-(e+2)}+1)\frac{1}{a}\|\cJ\|_{\sigma}\le\\ 
&&2^{-(e+2)}(\|\cT\|_{\sigma}+
2^{-(e+2)}+1)\le
2^{-(e+2)}(\|\cT\|+(2^{-(e+2)}+1)\|\cT\|)< 2^{-e}\|\cT\|.
\end{eqnarray*}
It is left to show that the bit complexity of finding $L(\cT)$ is $\tilde O\big((\tau+c+2nb\log_2(d-1)+c) d^{8n}\big)$.  This follows from Theorem \ref{cesnsingcten}.  Indeed observe that each entry by $\cT(a)$ is bounded by $2^{\tau+c+2}(d-1)^{2n(b+1)}$.
\end{proof}

We remark that we can find a strongly nonsingular $\cJ\in \rS^d\Z^n$ as follows:  We choose a tensor of the form given in the proof of Lemma \ref{varsnons} by choosing $a_2,\ldots,a_n\in\Z$ at random.

We now give a similar complexity result for an approximation of $\|\cT\|_{\sigma,\R}$:
\begin{theorem}\label{cesrealten} Let $d\ge 3$ be an integer.  Suppose that $\cT\in\rS^d\Z^n$, and each coordinate of $\cT$ is bounded above by $2^{\tau},\tau\in\N$.  
For a given  $e\in\N$ we can compute  $L(\cT)$ that satisfies the following conditions:  
\begin{enumerate}
\item Assume that $\cT$ is strongly nonsingular.  Then $|\|\cT\|_{\sigma,\R}-L(\cT)|\le 2^{-e} \|\cT\|_{\sigma,\R}$.
The bit complexity of computing $L(\cT)$ is $\tilde\cO\big((\tau+e)d^{4n}\big)$.
\item  Assume that $\cJ\in \rS^d\Z^n$ is a strongly nonsingular tensor that satisfies $\|\cJ\|\le 2^{c}, c\in\N$.  For a given $b\in\N$ we can compute $L(\cT)$ with probability
greater than $1-(d-1)^{-2nb}$ that satisfies $ |\|\cT\|_{\sigma,\R}-L(\cT)|\le 2^{-e} \|\cT\|$.  The bit complexity of computing $L(\cT)$ is $\tilde O\big((\tau+c+2nb\log_2 (d-1)+e) d^{4n}\big)$. With probability   
\end{enumerate}
\end{theorem} 
\begin{proof} We point out briefly the corresponding modifications of the proofs of Theorems \ref{cesnsingcten} and \ref{cessingcten} respectively.  Let $f(\x)=\cT\times\otimes^d\x$.  First observe that $\bar\bF=\bF$.  Hence any fixed point of $\bF$ is a fixed point $\bH$.  

First assume that $\cT$ is nonsingular.  Then $x_1$ coordinates of $(d-1)^n$ fixed points of $\bF$ are pairwise distinct.  We now find the Gr\"obner basis of the system $\bF(\x)-\x=\0$ with respect to the order $x_1\prec\cdots\prec x_n$.
It is of the form $p_1(x_1), x_2-p_2(x_1),\ldots,x_n-p_n(x_1)$, where $p_1,\ldots,p_n\in\R[x_1]$.  Here $\deg p_1=D$, where $D=(d-1)^n$.  Furthermore $p_1(x_1)$ has $D$ distinct roots.  Recall that $\deg p_i<D$ for $i>1$.  The set of the fixed points of $\bF$ is given by \eqref{solfixH}.

Recall that $\|\cT\|_{\sigma,\R}$ can be computed using \eqref{Sinftynrmformr}.  
Thus we need to determine $\mathrm{fix}_{\R}(\bF)$ with precision $2^{e+k}$, as in the proof of Theorem \ref{cesnsingcten}.  
This can be done as follows.  First consider $X_{\R}=X\cap\R^n$.   For each approximation of a root of $x_1$ with precision $2^{e+k}$, we check if the disk $|z-x_1|<2^{e+k}$ contains a real root of $p_1(x_1)$.  If yes, we replace $x_1$ by $\tilde x_1\in\R$.  Then a real approximation of the real fixed point is $(\tilde x_1,p_2(\tilde x_1),\ldots,p_n(\tilde x_n))\trans$.  Similarly we can find a real approximation of $\bar\omega_{d-2}X\cap \R^n$ for an even $n$.  (See Theorem \ref{theofoundthm}.)
Now we proceed as in the proof of Theorem \ref{cesnsingcten}.

For a general $\cT\in\rS^d\Z^n$ we repeat the arguments of the proof of Theorem \ref{cessingcten} taking into account the above results for a strongly nonsingular $\cT\in\rS^d\Z^n$.
\end{proof} 

We conclude this section with the following NP-hardness result for an arbitrary approximation of the spectral norm of a real or complex valued homogeneous quartic polynomial in $n$ variables.
\begin{theorem}\label{NPhardquart}  Let $A=[A_{i,j}]$ be an $n\times n$ nonzero symmetric matrix with $\{0,1\}$ entries and zero diagonal.  Set
\[f_A=\sum_{i=j=1}^n A_{i,j}x_i^2x_j^2.\]
\begin{enumerate}
\item Let $2e$ be the number of nonzero elements in $A$.  Then 
\[\|\mathbf{f}_A\|=\frac{\sqrt{2e}}{\sqrt{3}}, \quad \frac{\sqrt{2}}{\sqrt{3}}\le \|\mathbf{f}_A\|\le \frac{\sqrt{n(n-1)}}{\sqrt{3}}<n,\]
where $\|\mathbf{f}_A\|$ is the norm defined in Lemma \ref{isolem}.
\item Equality
\[\|f_A\|_{\sigma,\R}=\|f_A\|_{\sigma}=1-1/\kappa(A)\] 
holds, where $\kappa(A)$ is an integer in the set $[n]$.
\item  It is NP-hard to compute an approximation of $\|f_A\|_{\sigma}$ within relative  precision $\delta<\frac{1}{2n^2(n+1)}$ with respect to $\|\mathbf{f}_A\|$.
\end{enumerate}
\end{theorem}
\begin{proof} (1) Assume that $f_A$ corresponds to $\cS=[\cS_{k_1,k_2,k_3,k_4}]\in \rS^4\C^n$.  Then $\cS_{k_1,k_2,k_3,k_4}=1/3$ if the multiset $\{k_1,k_2,k_3,k_4\}$ is of the form $\{i,i,j,j\}$, $i<j$ and $A_{i,j}=A_{j,i}=1$.  (The number of nonzero  $\cS_{k_1,k_2,k_3,k_4}$ corresponding to  $\{i,i,j,j\}$ is $4!/(2!)^2=6$.)  Otherwise $\cS_{k_1,k_2,k_3,k_4}=0$.  Hence $\|\mathbf{f}_A\|=\|\cS\|=\sqrt{2e}/\sqrt{3}$.
The second inequality follows straightforward.

\noindent
(2)  The matrix $A$ is the adjacency matrix of the following simple undirected  graph $G=(V,E)$:
Here $V=[n]$ and the edge $(i,j)$ is in $E$ if and only if $A_{i,j}=1$.  Let $\kappa(A)$  be the cardinality of the maximal clique in $G$, (the clique number of $G$).  The equality $\|f_A\|_{\sigma,\R}=1-1/\kappa(A)$ is \cite[(8.2)]{FL18}.  As $\cS$ has nonnegative entries it follows that  $\|f_A\|_{\sigma,\R}=\|f_A\|_{\sigma}$ \cite{FL18}.

\noindent
(3)  Since $\|f_A\|_{\sigma}=1-1/\kappa(A)$, an approximation of  $\|f_A\|_{\sigma}$ within relative precision $\varepsilon<\frac{1}{2n^2(n+1)}$ with respect to $\|\mathbf{f}_A\|$ determines the clique number.  However, it is an NP-complete problem to determine the clique number of a simple graph \cite{Ka}.
\end{proof}

\section{Polynomial-time computability of spectral norm of symmetric $d$-qubits}\label{sec:dqubit}
In this section we improve the results of the previous section for the case $n=2$.  
We parametrize 
$\cS=[\cS_{i_1,\ldots,i_d}]\in \rS^d\F^2$ by 
 $\mathbf{s}=(s_0,\ldots,s_d)$ as follows:  $\cS_{i_1,\ldots,i_d}=s_k$ if exactly $k$ indices from the multiset
$\{i_1,\ldots,i_d\}$ are equal to $2$.  Note that exactly $d \choose k$ entries of $\cS$ are equal to $s_k$.  Hence
$\|\cS\|=\sqrt{\sum_{k=0}^d {d\choose k}|s_k|^2}$.

The following lemma is a restatement of some results of Lemma \ref{isolem}.
\begin{lemma}\label{qubspecnrmlem}  Let $\cS\in\rS^d\C^2$ and associate with $\cS$ the vector $\mathbf{s}=(s_0,\ldots,s_d)\trans\in \C^{d+1}$.  Denote
\begin{equation}\label{defphiz}
\phi(z)=\sum_{j=0}^d {d\choose j} s_j z^j.
\end{equation}
	Then
	\begin{enumerate}
		\item
		Let $f(\x)=\cS\times\otimes ^d\x$ and  $\cS\times\otimes^{d-1} \x=\mathbf{F}(\x)=(F_1(\x),F_2(\x))\trans$, where $\x=(x_1,x_2)\trans$.  Then
		\begin{eqnarray*}\label{deffx}
		&&f(\x)=\sum_{j=0}^d {d\choose j} s_j x_1^{d-j}x_2^j=x_1^d \phi(\frac{x_2}{x_1}),\\
		\label{formF0}
		&&F_1(\x)=\sum_{j=0}^{d-1} {d-1\choose j} s_jx_1^{d-1-j}x_2^j =\frac{1}{d}\frac{\partial f}{\partial x_1},\;
	F_2(\x)=\sum_{j=0}^{d-1} {d-1\choose j}s_{j+1} x_1^{d-j-1}x_2^j=\frac{1}{d}\frac{\partial f}{\partial x_2}.\label{formF1}
		\end{eqnarray*}
		\item For $\mathbf{s}=(0,\ldots,0,s_d)\trans$ we have $\|\cS\|_{\sigma,\F}=|s_d|$.
	\end{enumerate}	
\end{lemma}
\begin{proof} (1).  Use equalities \eqref{defpolx1} and \eqref{Fformulas}.
	
\noindent
(2). Assume that $\mathbf{s}=(0,\ldots,0,s_d)\trans$.  Then $\cS\times \otimes^d\x=s_d x_2^d$.  Hence $\|\cS\|_{\sigma,\F}=|s_d|$.
\end{proof}	

The next proposition studies the fixed points of $\bF$ and $\bH$ for the case $n=2$.
\begin{proposition}\label{fixptHFn=2}  Suppose that the assumptions and notations of Lemma \ref{qubspecnrmlem} hold.
Assume that $d\ge 3$ and
		\begin{equation}\label{Sassump}
		\mathbf{s}\ne (0,\ldots,0,s_d)\trans.
		\end{equation}
	\begin{enumerate}
		\item 
		Define polynomials $p(z), q(z)$ and the rational function $r(z)$ as follows
		\begin{eqnarray*}\label{fzfor}
		p(z)=\sum_{j=0}^{d-1} {d-1\choose j}s_{j+1} z^j,\; q(z)=\sum_{j=0}^{d-1} {d-1\choose j} s_jz^j,\; r(z)=\frac{p(z)}{q(z)}.\;
		\end{eqnarray*} 
		Then
		\begin{equation}\label{relpqrphi}
		p(z)=\frac{1}{d}\phi'(z),\; q(z)=\phi(z)-\frac{1}{d}z\phi'(z), \;r(z)=\frac{\phi'(z)}{d\phi(z)-z\phi'(z)}.
		\end{equation}
\item
		Suppose that $\bF(\x)=(x_1,x_2)\trans\ne \0$.  First assume that $x_1\ne 0$.  Let $z=x_2/x_1$.  Then
		\begin{equation}\label{fixqub1}
		z q(z)-p(z)=0 \textrm{ and } q(z)\ne 0.
		\end{equation}
		Vice versa, for each $z\in\C$ satisfying the above conditions there are exactly $d-2$ fixed  points of $\bF$ of the form $(x_1,x_1 z)$, where $x_1$ satisfies $x_1^{d-2}=\frac{1}{q(z)}$.  
		
Second assume that $x_1=0$.  Then $s_{d-1}=0$ and  $s_d\ne 0$.
		Vice versa, if $s_{d-1}=0$ and  $s_d\ne 0$ then there are exactly $d-2$ nonzero fixed points of the form $(0,x_2)$, where $x_2$ satisfies $x_2^{d-2}=1/s_d$.
		\item  Suppose that $\bF(\x)=(\bar x_1, \bar x_2)\ne \0$.  First assume that $x_1\ne 0$.  Let $z=x_2/x_1$.  Then
	\begin{equation}\label{antifixqub1}
		\bar z q(z)-p(z)=0
		\end{equation}
and  $q(z)\ne 0$.
Vice versa, for each $z\in\C$ satisfying the above conditions there are exactly $d$ antifixed  points of $\bF$ of the form $(x_1,x_1 z)$, where $x_1$ satisfies $x_1=\frac{\zeta}{|q(z)|^{1/(d-2)}}$ and $\zeta^d=1$. 

			Second assume that $x_1=0$.  Then $s_{d-1}=0$ and $s_d\ne 0$.   
		Vice versa, if $s_{d-1}=0$ and  $s_d\ne 0$  then there are $d$ nonzero fixed points of the form $(0,\frac{\zeta}{|s_d|^{1/(d-2)}})$ and $\zeta^d=1$.  
		\item Let
		\begin{eqnarray*}\label{defbarpqf}
		\bar p(z)=\sum_{j=0}^{d-1} {d-1\choose j}\bar s_{j+1} z^j, \;\bar q(z)=\sum_{j=0}^{d-1} {d-1\choose j} \bar s_jz^j,\;\bar r(z)=\frac{\bar p(z)}{\bar q(z)},\;g(z)=\bar r (r(z)).
		\end{eqnarray*}
		Then $g(z)=\frac{u(z)}{v(z)}$, where
		\begin{eqnarray}\label{defuz}
		u(z)= \sum_{j=0}^{d-1} {d-1\choose j}\bar s_{j+1} 
		\big(\sum_{k=0}^{d-1} {d-1\choose k} s_{k+1} z^k\big)^{j}\big(\sum_{k=0}^{d-1} {d-1\choose k}  s_kz^k\big)^{d-1-j},\\ 
		\label{defvz}
		v(z)= \sum_{j=0}^{d-1} {d-1\choose j}\bar s_{j} 
		\big(\sum_{k=0}^{d-1} {d-1\choose k} s_{k+1} z^k\big)^{j}\big(\sum_{k=0}^{d-1} {d-1\choose k} s_kz^k\big)^{d-1-j}.\quad
		\end{eqnarray} 
Suppose that $\bF(\x)=(\bar x_1, \bar x_2)\trans \ne \0$ as in (3).  Assume that  $x_1\ne 0$.  Then each solution of \eqref{antifixqub1} satisfies
		\begin{equation}\label{poleqfixpoint}
		zv(z)-u(z)=0
		\end{equation}
		and $v(z)\ne 0$. 
		This is a polynomial equation of degree at most $(d-1)^2+1$.   
		\end{enumerate}
\end{proposition}
\begin{proof}
(1)  The assumption  $\mathbf{s}\ne (0,\ldots,0,s_d)\trans$ is equivalent to the assumption that the polynomial $q(z)$ is not zero identically.  Assume that $x_1\ne 0$. Let $z=\frac{x_2}{x_1}$.  Lemma \ref{qubspecnrmlem} yields:
	$\frac{F_2(x_1,x_1)}{x_1^{d-1}}=p(z)$ and $\frac{F_1(x_1,x_1)}{x_1^{d-1}}=q(z)$.
	Recall that
	\begin{eqnarray*}
		dF_2=\frac{\partial f}{\partial x_2}=\frac{\partial (x_1^{d}\phi(\frac{x_2}{x_1}))}{\partial x_2}=x_1^{d-1}\phi'(\frac{x_2}{x_1}),\\
		dF_1=\frac{\partial f}{\partial x_1}=\frac{\partial (x_1^{d}\phi(\frac{x_2}{x_1}))}{\partial x_1}=dx_1^{d-1}\phi(\frac{x_2}{x_1})-x_2x_1^{d-2}\phi'(\frac{x_2}{x_1}).
	\end{eqnarray*}
	These equalities yield (\ref{relpqrphi}).

\noindent
(2)  Assume that $\bF(\x)=(x_1,x_2)\trans \ne \0$.  
	Suppose first that $x_1\ne 0$.  Then 
$x_1^{d-1} q(z)=F_1(\x)=x_1$ and $x_1^{d-1}p(z)=F_2(\x)=x_2$.  As $x_1\ne 0$ it follows that $q(z)\ne 0$.
Divide the second equality by the first one to deduce  \eqref{fixqub1}.  Note that $x_1^{d-2}=\frac{1}{q(z)}$.

Assume that $z\in\C$ satisfies \eqref{fixqub1}.  Suppose furthermore that $x_1^{d-2}=\frac{1}{q(z)}$.  Let $\x=(x_1,x_1 z)\trans$.  Then 
\[F_1(\x)=x_1^{d-1} q(z)=x_1, \quad F_2(\x)=x_1^{d-1}p(z)=x_1^{d-1} zq(z)=x_1^{d-1} (x_2/x_1)q(z)=x_1^{d-2}q(z)x_2=x_2.\]
Hence each $z$ that satisfies \eqref{fixqub1} gives rise to exactly $d-2$ distinct nonzero fixed points of $\bF$. 

Assume that $x_1=0$.  Hence $x_2\ne 0$.  Note that $F_1(\x)=s_{d-1}x_2^{d-1}$.
As $F_1(\x)=x_0=0$ we deduce that $s_{d-1}=0$.  Clearly, $F_2(\x)=s_d x_2^{d-1}=x_2$.  As $x_2\ne 0$ it follows that $s_d\ne 0$.  

Suppose that $s_{d-1}=0$ and  $s_d\ne 0$.  Then all nonzero fixed points of $\bF$ of the form $(0,x_2)\trans$ are exactly those $x_2$ satisfying $x_2^{d-2}=1/s_d$.

\noindent
(3)  The proof of this part is very similar to the part (2) and we leave it to the reader.

\noindent
(4)  From the definitions of  $p(z),q(z),\bar p(z), \bar q(z), \bar r(z) $ and $g(z)$  we deduce straightforward the identities (\ref{defuz}) and (\ref{defvz}).
	Let
	\[u(z)=\sum_{k=0}^{(d-1)^2} u_kz^k, \quad v(z)=\sum_{k=0}^{(d-1)^2} v_kz^k.\]
	Then
	\begin{eqnarray*}\label{uvdfor}
	u_{(d-1)^2}=\sum_{j=0}^{d-1} {d-1\choose j}\bar s_{j+1}s_{d}^j s_{d-1}^{d-1-j},\;
	v_{(d-1)^2}= \sum_{j=0}^{d-1} {d-1\choose j}\bar s_{j}s_{d}^j s_{d-1}^{d-1-j}.
	\end{eqnarray*}
	Hence the polynomial $zv(z)-u(z)$ is of degree at most $(d-1)^2+1$.  Suppose $\x=(x_1,x_2)\trans\in\mathrm{fix}(\bH), x_1\ne 0$: 
$\overline{\mathbf{F}}(\bF(\x))=\x$.
	Since we assumed that $x_1\ne 0$ it follows that 
	\[u(z)=\frac{\bar\bF_2(\bF(\x))}{x_1^{(d-1)^2}}, \quad v(z)=\frac{\bar\bF_1(\bF(\x))}{x_1^{(d-1)^2}}=x_1^{-(d-1)^2 +1}\ne 0.\]
	Therefore
	\[\bar r( r(z))=\frac{\bar\bF_2(\bF(\x))x_1^{-(d-1)^2}}{\bar\bF_1(\bF(\x))x_1^{-(d-1)^2}}=\frac{x_2}{x_1}=z.\]
	This yields (\ref{poleqfixpoint}), which is a polynomial equation of degree at most $(d-1)^2 +1$.
\end{proof}			

The following theorem gives much more efficient way to compute the spectral norm of symmetric qubits than the general methods suggested in \S\ref{sec:dnqudit}. 
\begin{theorem}\label{computdqubspecnrm}  Let $\cS\in\rS^d\C^2$, $d>2$ and associate with $\cS$ the vector $\mathbf{s}=(s_0,\ldots,s_d)\trans\in \C^{d+1}$.  Let $\phi(z)$ be given by \eqref{defphiz}.  Assume the notations of Proposition \ref{fixptHFn=2}.  Then
	\begin{enumerate}		
		\item The polynomial $zv(z)-u(z)$ is a zero polynomial if and only if one of the following conditions hold.  Either 
		\begin{eqnarray*}\label{geqiden}
		\mathbf{s}=A(\delta_{1(k+1)},\ldots,\delta_{(d+1)(k+1)})\trans \textrm{ for } k\in[d-1], (f(\x)=A{d\choose k}x_1^{d-k}x_2^k),
		\end{eqnarray*}
		where $A$ is a nonzero scalar constant and $\delta_{ij}$ is Kronecker's delta function.  For this $\cS\in\rS^d\F^2$ we have
		\begin{equation}\label{geqidenSform}
		\|\cS\|_{\sigma,\F}=|A|{d \choose k} \big(1-\frac{k}{d}\big)^{\frac{d-k}{2}}\big(\frac{k}{d}\big)^{\frac{k}{2}}.
		\end{equation}
		Or $\cS$ has corresponding $\phi$ given by
		\begin{eqnarray}\label{geqiden1}
		&&\phi(z)=A (z+a)^p(z+b)^{d-p},\quad A\ne 0,\\ 
		&&a=e^{-\theta\bi}c,\; b=-e^{-\theta\bi}c^{-1},\;c\in\R\setminus\{0\}, \; \theta\in\R,\;p\in[d-1].\notag
		\end{eqnarray}
		Assume that $\cS\in\rS^d\C^2$ corresponds to $\phi$ is of the form (\ref{geqiden1}).  Then $\|\cS\|_{\sigma}$ can be computed to an arbitrary precision as explained in \S\ref{sec:excepcase}.
		\item  Let
		\begin{eqnarray*}
		R_1=\{z\in\C,\;zv(z)-u(z)=0\}\label{defR1}, \quad R_1'=R_1\cap \R.
		\end{eqnarray*}	 	
		Suppose that $\cS\in\rS^d\C^2$ is not of the form given in (1).  (Hence the set $R_1$ is finite.)  Then $\|\cS\|_{\sigma}$ has the following maximum characterization:
		\begin{eqnarray}
		\label{specnrmSform1}
		\|\cS\|_{\sigma}= \max\left\{|s_d|, \max\{\frac{|\phi(z)|}{(1+|z|^2)^{\frac{d}{2}}}, z\in R_1\}\right\}.
		\end{eqnarray}
		\item  Assume that $\cS\in\rS^d\R^2$.  Let
		\begin{eqnarray*}\label{defR}
		R=\{z\in\C, \;z q(z)-p(z)=0\},\quad R'=R\cap\R
		\end{eqnarray*}	 
		Then $zq(z)-p(z)$ is a zero polynomial if and only if $d$ is even and $\phi(z)=A(z^2+1)^{d/2}$.  In this case $\|\cS\|_{\sigma,\R}=|A|$.  
		
Assume that $\phi(z)$ is not of the form $A(z^2+1)^{d/2}$.		Then $\|\cS\|_{\sigma,\R}$ has the following maximum characterizations: 
		\begin{eqnarray}\label{realforspecnrm}
		\|\cS\|_{\sigma,\R}=\max\left\{|s_d|, \max\{\frac{|\phi(z)|}{(1+|z|^2)^{\frac{d}{2}}}, z\in R'\}\right\}.
		\end{eqnarray}
		
	\end{enumerate}
\end{theorem}
\begin{proof}  
	
	\noindent
	(1)  The analysis of the exceptional cases is given in \S\ref{sec:excepcase}.
	
	\noindent
	(2)  Assume that $\mathbf{s}=(s_0,\ldots,s_d)\trans\in\C^{d+1}$ is the vector corresponding to $\cS=[\cS_{i_1,\ldots,i_d}]\in\rS^d\C^2$.  Let $\omega$ be the righthand side of \eqref{specnrmSform1}.  We claim that Banach's characterization \eqref{Banchar} yields that $\omega \le \|\cS\|_{\sigma}$.  Indeed, $\|\cS\|_{\sigma}\ge |\cS\times(\otimes^d(0,1)\trans)|=|\cS_{2,\ldots,2}|=|s_d|$.  Assume now that $\x=(x_1,x_2)\trans$ and $x_1\ne 0$.  Let $z=x_2/x_1$.  Then $|\cS\times(\otimes^d \x)|/\|\x\|^{d}=|\phi(z)|/(1+|z|^2)^{d/2}$.  Hence $\omega \le \|\cS\|_{\sigma}$.  As any maximal point of $|\cS\times(\otimes^d \x)|$ on $\rS(2,\C)$ is either 
	$\zeta(0,1)\trans, |\zeta|=1$, or of the form $\zeta(1,z),\zeta\in\C\setminus\{0\}$, where $z\in R_1$ we deduce that $\omega =\|\cS\|_{\sigma}$.
	
\noindent
(3)	 Assume that $zq(z)-p(z)$ is identically zero. The equalities \eqref{relpqrphi} yield that 
\[z\phi=(1/d)(z^2+1)\phi'(z)\Rightarrow (d/2) (\ln (z^2 +1))'=(\ln \phi)'\Rightarrow   
\phi = A(z^2+1)^{d/2}.\]
As $\phi$ is a polynomial, it follows that $d$ is even.   Note that $\cS\times (\otimes ^d\x)=A(x_0^2+x_1^2)^{d/2}$.  Hence $\|\cS\|_{\sigma,\R}=|A|$.

Assume now $\phi$ is not of the form $A(z^2+1)^{d/2}$.  Then the arguments of the proof of part (2) yield the equality \eqref{realforspecnrm}.		
\end{proof}

We now give the complexity of finding the spectral norm of  $\cS\in\rS^d\Z[\bi]^2$ or $\cS\in \rS^d\Z^2$.  
\begin{theorem}\label{complexaynqub}  Let $\cS\in\rS^d\Z[\bi]^2$, $d>2$ and associate with $\cS$ the vector $\mathbf{s}=(s_0,\ldots,s_d)\trans\in \Z[\bi]^{d+1}$.  
Assume that $|s_{i-1}|\le2^{\tau}, i\in[d+1]$ for some $\tau\in\N$. Let $\phi(z)$ be given by \eqref{defphiz}.  
Assume the notations of Proposition \ref{fixptHFn=2}.   For a given $e\in\N$ we can compute  a rational $L(\cS)$ satisfying 
\begin{eqnarray*}\label{LcSineqm}
 |\|\cS\|_{\sigma,\F}-L(\cS)|\le 2^{-e} \|\cS\|_{\sigma,\F}
\end{eqnarray*}
under the following conditions:
\begin{enumerate}
\item
The tensor $\cS\in \rS^d\Z[\bi]^2$ does not satisfy condition (1) of Theorem \ref{computdqubspecnrm}.
The bit complexity of computation of $L(\cS)$ for $\F=\C$ is  $\tilde\cO(d^2(d^4\max(d^2,\tau)+e))$.
\item
 The tensor $\cS$ has real integer entries.  The bit complexity of computation of $L(\cS)$ for $\F=\R$ is  $\hat\cO(d(d^2\max(d,\tau)+e))$.
\end{enumerate}
\end{theorem} 
\begin{proof} (1)  The assumption that condition (1) of Theorem \ref{computdqubspecnrm}  does not hold is equivalent to the assumption that the polynomial $h(z)=zv(z)-u(z)$ is a nontrivial polynomial, whose degree $D$ satisfies $D\le (d-1)^2+1<d^2$.  (See part  (4) of Proposition \ref{fixptHFn=2}.)   Observe that $h(z)=zv(z)-u(z)\in (\Z[\bi])[z]$, and  its coefficients are bounded by $2^{2\tau+2d}$.
As in the proof of Theorem \ref{cesnsingcten},
we approximate the roots of $h(z)$ with precision $2^{-(e+k)}$, where $k\in\N$ is specified later.  Now we repeat the arguments of Theorem \ref{cesnsingcten} by replacing the characterization \eqref{charcTspecnrm} with \eqref{specnrmSform1}.
This yields the estimate $ |\|\cS\|_{\sigma}-L(\cS)|\le 2^{-e} \|\cS\|_{\sigma}$.  The bit complexity estimate $\tilde\cO(d^2(d^4\max(d^2,\tau)+e))$ follows from \cite{NR96}.
(See Appendix 2.)

\noindent (2)   We now repeat the arguments of the proof of (1) with the following modifications.  Assume that $\phi(z)$ corresponds to $\cS\in \rS^d \Z^n$.  Suppose first $\phi(z)=A(z^2+1)^{d/2}$, where $d$ is even.  Then $\|\cS\|_{\sigma,\R}=|A|$ and let $L(\cS)=|A|$.   Assume now that $\phi(z)\ne A(z^2+1)^{d/2}$.  Then $h(z)=zq(z)-p(z)$ is a nonzero polynomial of degree $D\le d$.   We approximate the roots of $h(z)$ to approximate the set $R'$ as given in part (3) of Theorem \ref{computdqubspecnrm}.
We use characterization \eqref{realforspecnrm} to compute $L(\cS)$.
\end{proof} 
\section{The exceptional cases}\label{sec:excepcase}
\subsection{Analysis of the exceptional cases}\label{sec:analexcas}
In this subsection we discuss part (1) of Theorem \ref{computdqubspecnrm}. 
Assume that $g(z)=\bar r(r(z))=z$ identically.  Recall that $r(z)$ can be viewed as a holomorphic map of the Riemann sphere.  The degree of this map is $\delta\in\N$ since $g$ is not a constant map.  Hence the degree of the map $g$ is $\delta^2$.  Since $g$ is the identity map, its degree is $1$, it follows that $\delta=1$ and $r(z)$ is a M\"obius map:
\[r(z)=\frac{az+b}{cz+d}, \quad ad-bc\ne 0.\]
Use the formula for $r(z)$ in (\ref{relpqrphi}) to deduce
\[\frac{1}{d}(\log \phi(z))'=\frac{1}{d}\frac{\phi'(z)}{\phi(z)}=\frac{az+b}{cz+d +z(az+b)}.\]
Let $l$ be the number of distinct roots of $\phi(z)$.  Then the logarithmic derivative of $\phi(z)$ has exactly $l$ distinct poles.  Comparing that with the above formula of the logarithmic  derivative of $\phi$, we deduce that $\phi$ has either one (possibly) multiple root or two distinct roots.  

First assume that $\phi(z)$ has one root of multiplicity $k$: $\phi(z)=A(z+a)^k$ for $k\in[d]$.  Then 
\[r(z)=\frac{k}{(d-k)z+d a}, \quad \bar r(z)=\frac{k}{(d-k)z+d \bar a}.\]
In this case $g(z)\equiv z$ if and only if $k\in[d-1]$ and $a=0$.
Clearly, if $\phi(z)=Az^k$, where $A\ne 0$, then $g(z)\equiv z$.
In this case $\cS\times\otimes^d \x=A{d\choose k} x_1^{d-k}x_2^k$.  To find $\|\cS\|_{\sigma,\F}$ we need to maximize $|A|{d\choose k}|x_1|^{d-k}|x_2|^k$ subject to
$|x_1|^2+|x_2|^2=1$.  The maximum is obtained for 
$|x_1|^2=1-\frac{k}{d}, |x_2|^2=\frac{k}{d}$.  This proves (\ref{geqidenSform}).

Second assume that $\phi(z)$ has two distinct zeros: $\phi(z)=A(z+a)^p(z+b)^q$, where $a\ne b$, $p,q\in\N$ and $p+q\le d$.  Then
\[r(z)=\frac{(z+a)^{p-1}(z+b)^{q-1}\big((p+q)z+pb+qa\big)}{(z+a)^{p-1}(z+b)^{q-1}\big(d(z+a)(z+b)-z\big((p+q)z+pb+qa\big)\big)}.\]
Suppose first that $p+q<d$.  In order that $r(z)$ will be a M\"obius transformation
we need to assume that $(p+q)z+pb+qa$ divides $(z+a)(z+b)$.  This is impossible,
since $\phi'$ has exactly $p+q-2$ common roots with $\phi$.  Hence we are left with the case $p+q=d$.  In this case
\begin{eqnarray*}\label{rxfroexcase}
r(z)=\frac{dz+\alpha}{\beta z+dab}, \quad \alpha=pb+qa, \beta=d(a+b)-\alpha, p+q=d.
\end{eqnarray*}
The assumption that $g(z)\equiv z$ is equivalent to the following matrix equality
\begin{eqnarray*}\label{defmatA}
\bar B B=\gamma^2 I_2, \quad B=\left[\begin{array}{cc}d&\alpha\\\beta&dab\end{array}\right] \quad \gamma\ne 0.
\end{eqnarray*}
Taking the determinant of the above identity we deduce that $\gamma^4=|\det B|^2>0$.  So $\gamma^2=\pm \tau^{-2}$ for some $\tau>0$.
Let $C=\tau B$.  Suppose first that $\gamma^2=\tau^{-2}$.  Then $\bar C$ is the inverse of $C$.  So $\det C=\delta, |\delta|=1$.  Then
\begin{eqnarray*}\label{analeq}
dab=\delta d,\; d=\delta d\bar a\bar b,\;-\alpha=\delta\bar \alpha, -\beta=\delta \bar\beta.
\end{eqnarray*}
Hence $ab=\delta$.  

We next observe that if we replace $\phi(z)$ with $\phi_{\theta}(z):=\phi(e^{\theta\bi }z)$ for any $\theta\in\R$ we will obtain the following relations
\[p_{\theta}(z)=e^{\theta \bi}p(e^{\theta \bi}z),\; q_{\theta}(z)=q(e^{\theta \bi}z),\; \overline{p_{\theta}}(z)=e^{-\theta \bi}p(e^{-\theta \bi}z),\; \overline{q_{\theta}}(z)=\bar q(e^{-\theta \bi}z).\]
Let 
\[r_{\theta}(z)=\frac{p_\theta(z)}{q_\theta(z)}, \;\overline{r_\theta}(z)=\frac{\overline{p_\theta}(z)}{\overline{q_\theta}(z)},\; g_\theta(z)=\overline{r_\theta}(r_\theta(z)).\]
A straightforward calculation shows that $g_\theta(z)=z$ for all $\theta\in\R$.  Note the two roots of $\phi_\theta(z)$ are $-ae^{-\theta\bi}, -be^{-\theta\bi}$.  Hence we can choose $\theta$ such that $\delta=-1$.  Assume for simplicity of the exposition this condition holds for $\theta=0$, i.e., for $\phi$.  So $\alpha$ and $\beta$ are real.  In particular $a+b$ is real.  So $b=-a^{-1}$ and $a-a^{-1}$ is real.  Hence $a$ is real and also $b$ is real.

Suppose now that $\gamma^2=-\tau^2$.  Then $\bar B=-B^{-1}$. So $\det B=\delta, |\delta|=1$.  Then
\begin{equation}\label{analeq1}
dab=-\delta d,\; d=-\delta d\bar a\bar b,\;\alpha=\delta\bar \alpha, \beta=\delta \bar\beta.
\end{equation}
By considering $\phi_\theta$ as above we may assume that $\delta=1$ which gives again that $a\in\R\setminus\{0\}$ and $b=-a^{-1}$.  This proves (\ref{geqiden1}).

Vice versa, assume that $\phi(z)$ is of the form (\ref{geqiden1}), where $a\in\R\setminus\{0\}$ and $b=-\frac{1}{a}$.  We claim that $\bar r(r(z))\equiv z$.
The above arguments show that $r(z)=(dz+\alpha)/(\beta z+dab)$.  Furthermore
\begin{eqnarray*}\label{excforalphbet}
ab=-1,\;\alpha=(d-p)a -\frac{p}{a}, \beta=pa -\frac{d-p}{a}, \; p\in[d-1].
\end{eqnarray*}
Consider the matrix $B$ as given above. Note that the trace of this matrix is zero.  We claim that $B$ is not singular.  Indeed
\begin{eqnarray*}\notag
&&\det B=-(d^2 +\alpha\beta)=-(d^2 - (d-p)^2 -p^2 +p(d-p)(a^2+a^{-2}))=\\
&&-p(d-p)(a+a^{-1})^2.\label{formdetA}
\end{eqnarray*}
Hence $B$ has two distinct eigenvalues $\{\gamma,-\gamma\}$ and is diagonalizable.
As $B$ is a real matrix we get that $\bar B B= B^2=\gamma^2 I_2$.  Therefore 
$\bar r(r(z))\equiv z$.  As $\bar r_\theta(r_\theta(z))\equiv z$ for each $s\in\R$ we deduce that for each $\phi$ of the form (\ref{geqiden1}) $zq(z)-p(z)$ is a zero polynomial.
\subsection{Computation of $\|\cS\|_{\sigma}$ in the exceptional cases}\label{sec:compexcas}
Assume now that $\rS_\theta\in\rS^d\C^2$ is induced by $\phi$ of the form (\ref{geqiden1}).  Then
\begin{eqnarray*}\label{formFxcom}
\cS_{\theta}\times\otimes^d\x=A(x_2+ce^{-\theta\bi}x_1)^p(x_2-c^{-1}e^{-\theta\bi}x_1)^{d-p}
, \; c\in\R\setminus\{0\},\theta\in\R,p\in[d-1],
\end{eqnarray*}
Clearly, $\|{\cS}_{\theta}\|_{\sigma}=\|\cS_0\|_{\sigma}$.  Thus it is enough to find $\|\cS_0\|_{\sigma}$.  For simplicity of notation we let $\cS=\cS_0$.  

We now suggest the following simple approximations to find $\|\cS\|_{\sigma}$ using the case (2) of Theorem \ref{computdqubspecnrm}.   It is enough to assume that $A=1$, i.e., $\cS$ corresponds to $\phi(z)=(z+c)^p(z-c^{-1})^{d-p}$.  
Let $\omega$ be a rational in the interval $(0,1)$.  Set  $\phi(z,\omega)=\phi(z)+\omega$.
Let $\cS(\omega)\in\rS^d\R^2$ be the symmetric tensor corresponding to $\phi(z,\omega)$.  Then $S(\omega)\times(\otimes^d \x)=(x_2+cx_1)^p(x_2-c^{-1}x_1)^{p-d}+\omega x_1^d$. 
By choosing $\x=(0,1)\trans$ we deduce that $\|\cS(\omega)\|_{\sigma}\ge 1$.   It is straightforward to show that 
$\|\cS(\omega)-\cS\|_{\sigma}= \omega$.  As $\omega\le \omega\|\cS(\tau)\|_{\sigma}$ we obtain
\[\|\cS\|_{\sigma}\in [\|\cS(\omega)\|_{\sigma}(1-\omega),\|\cS(\omega)\|_{\sigma}(1+\omega)].\]  

Observe next that $\phi'(z,\omega)=\phi'(z)$.  Hence a common root of $\phi(z,\omega)$ and 
$\phi'(z,\omega)$ can not be a common root of $\phi(z)$ and $\phi'(z)$.  Therefore
$\phi(z,\omega)$ and $\phi'(z,\omega)$ can have at most one common root.  As $d\ge 3$ we deduce that
$\cS(\omega)$ does not satisfy the assumptions of part (2) of Theorem \ref{computdqubspecnrm}.  

Assume that $c\in \Q$.  
Let $\delta$ be rational in the interval $(0,1)$ and assume that $\omega=\delta/4$.  Use the case (1) of Theorem \ref{complexaynqub} to find an approximation $L(\cS(\omega))$ that satisfies $|\|\cS(\omega)\|_{\sigma}-L(\cS(\omega))|\le (\delta/4)\|\cS(\omega)\|$.  Let us take $L(\cS(\omega))$ to be an approximation for $\|\cS\|$.  Then
\begin{eqnarray*}
|\|\cS\|_{\sigma} - L(\cS(\omega))|\le| \|\cS\|_{\sigma} - \|\cS(\omega)\|_{\sigma}|+|\|\cS(\omega)\|_{\sigma} - L(\cS(\omega))|\le (\omega +\delta/4)\|\cS(\omega)\|_{\sigma}\le \frac{\delta}{2}\frac{1}{1-\delta/4}\|\cS\|_{\sigma}<\delta \|\cS\|_{\sigma}.
\end{eqnarray*}
Thus we found a rational approximation of $\|\cS\|_{\sigma}$ that satisfies the inequality of Theorem \ref{complexaynqub} for $\F=\C$.  The complexity of $L(\cS(\omega))$ is given in case (1) of 
Theorem \ref{complexaynqub}.
\section{Numerical examples}\label{sec:numerex}
In this section we give some numerical examples of applications of 
Theorem \ref{computdqubspecnrm} for the qubit  cases, and of Theorem \ref{theofoundthm} for $n>2$.  
In many examples here $f(\x)$ is a sum of two monomials.  In this case we show that to compute the complex spectral norm of $f$ one can assume that the coefficients of the two monomials
are nonnegative.  Equivalently, we can assume the symmetric tensor $\cS$ has nonnegative entries. 
 Hence $\|\cS\|_{\sigma}=\|\cS\|_{\sigma,\R}=\|f\|_{\sigma}$ \cite{FL18}.
 However, if $f$ has real coefficients then it may happen that $\|f\|_{\sigma,\R}<\|f\|_{\sigma}$.  See Example 2 below.
\begin{lemma}\label{sum2mon}  Let $f\in\rP(d,n,\C)$, where $d,n\ge 2$, and assume that $f=a\x^{\bj}+b\x^{\bk}$, where $\bj\ne \bk$.  Denote $g(\x)=|a|\x^{\bj}+|b|\x^{\bk}$.
Then $\|f\|_{\sigma}=\|g\|_{\sigma}=\|g\|_{\sigma,\R}=g(\y)$ for some $\y\ge \0, \|\y\|=1$.
\end{lemma}
\begin{proof} Clearly, it is enough to assume that $a\ne 0$ and $b\ne 0$.  Let $\x=(x_1,\ldots,x_n)\trans$ and denote $\x_+=(|x_1|,\ldots,|x_n|)\trans$.  Note that $\|\x\|=\|\x_+\|$.
Assume that $\|\x\|=1$.  As $|f(\x)|,|g(\x)|\le g(\x_+)$ it follows that $\|f\|_{\sigma}\le \|g\|_{\sigma}=g(\y)$, for some $\y\ge 0$ and $\|\y\|=1$.  Hence $\|g\|_{\sigma}=\|g\|_{\sigma,\R}$.
Observe next that $f(\x)=\x^{\bl}(a\x^{\bj'}+b\x^{\bk'})$.  We can choose $\bl\ge 0$ so that the monomials $\x^{\bj'}$ and $\x^{\bk'}$ do not have a common variable.
Assume that $\z=(z_1,\ldots,z_n)\trans$ such that $\z_+=\y$.  Then it is possible to choose the arguments of $z_1,\ldots,z_n$ such that $|a\z^{\bj'}+b\z^{\bk'}|=|a|\y^{\bj'}+|b|\y^{\bk'}$.  Hence $|f(\z)|=g(\y)=\|g\|_{\sigma}$.  Therefore $\|f\|_{\sigma}=\|g\|_{\sigma}$.  
\end{proof}

In some of the examples we discuss, we modified the examples of $f$, that satisfy  conditions of Lemma \ref{sum2mon}, by considering $f_e$:
\begin{eqnarray}\label{ftform}
&&f_e(\x)=\sqrt{1-|e|^2}f(\x) + eh(\x), \quad e=t\omega,\\
&&t=0,1/5,1/4/,1/3,1/2,1,\quad \omega=1,-1,i, 1/2 +i\frac{\sqrt{3}}{2},-1/2 +i\frac{\sqrt{3}}{2},\notag\\
&&f=a\x^{\bj}+b\x^{\bk}, \; h=c\x^{\bl},\; \bj\ne \bk, \bj\ne \bl,\bk\ne \bl,\; a,b,c>0, \|f\|=\|h\|=1.\notag
\end{eqnarray}
Note that $\|f_e\|=1$, and $f_0=f, f_1=h$.  Hence, when we give our results for $\|f_e\|_{\sigma}$ we give separately the values of $\|f\|_{\sigma},\|h\|_{\sigma}$, and $\|f_e\|_{\sigma}$ for $t=1/5,1/4/,1/3,1/2$ and the above values of $\omega$.

First we discuss the symmetric $d$-qubits. If the 
 $d$-qubits are real we  compute both the complex and real spectral norms of a given tensor $\cS$. 
In Theorem \ref{computdqubspecnrm}, we showed that the spectral norm of a given $d$-qubit can be found by solving the polynomial equation $zv(z)-u(z) =0$ (see (\ref{poleqfixpoint})), which is of degree of at most $(d-1)^2+1$.  In what follows that we assume that the polynomial $zv(z)-u(z)$ is not a zero polynomial.  (That is, we are not dealing with the exceptional cases that are discussed in \S\ref{sec:excepcase}.)  We use formula (\ref{specnrmSform1}) to find $\|\cS\|_{\sigma}$, and formula (\ref{realforspecnrm}) to find $\|\cS\|_{\sigma,\R}$.

 For $n>2$ we use Theorem \ref{theofoundthm} to compute $\|f\|_{\sigma}$.  Thus we need to assume that $\mathrm{fix}(\bH)$ is finite, or if all the monomials of $f$ have nonnegative coefficents then the set of the real fixed points of $\bF$ is finite.

In this paper we use Bertini \cite{BHSW06} (version 1.5, released in 2015), which is a well developed software to find all the complex solutions of a given polynomail system.    It is worth noting that, in contrast to the theoretical results of this paper, the output of Bertini is not certified and may be incorrect by an unbounded error if the software jumps between solutions when tracking the homotopy paths.

All the computation are implemented with Matlab R2018b on a MacBook Pro 64-bit OS X (10.12.6) system with 16GB memory and 2.9 GHz Intel Core i7 CPU. In the display of the computational results, only four decimal digits are shown. The default parameters in Bertini are used to solve the polynomial equation $zv(z)-u(z) =0$. 
Since all examples only take few seconds we will not show the computing time. 

In our examples  all polynomials correspond to the symmetric tensors $\cS\in \rS^d\C^n$ of Hilbert-Schmidt norm one.  Furthermore, if $\cS$ has nonnegative entries then $\|\cS\|_{\sigma}=\|\cS\|_{\sigma,\R}$ \cite{FL18}.

Assume that $d=2$.  Then $|s_d|=|\cS_{2,\ldots,2}|$, and if $\cS$ has nonnegative entries  then one can use part (2) of Theorem \ref{computdqubspecnrm}.
In these examples we find the degree of the polynomial $zv(z)-u(z)$, the number of real and complex roots, and the number of roots that fail to satisfy (\ref{antifixqub1}).
\subsection{Three examples of symmetric $3$-qubits}\label{subsec:first3q}
These examples are interesting to us, and some of them are discussed in other papers.  
\setcounter{theorem}{0}
\begin{example}\emph{\cite{Nie14}} \label{exm:1:d:qubit} Let
$f=0.3104 x_1^3 - 1.4598x_1^2x_2 - 0.6558 x_1 x_2^2+0.2235 x_2^3$.
	The polynomial $zv(z)-u(z)$ has degree $5=2^2+1$. It has 5 roots, 3 of them are real and the other 2 are complex.  We have $R=R_1$, $R'=R_1'$. Then
	\[	\|f\|_{\sigma~~} \approx 0.7027,\quad
	\|f\|_{\sigma,\R} \approx 0.6205.\]
\end{example}

\begin{example}  \emph{\cite{FL18}}  \label{exm:1:B:d:qubit} Let $f=\frac{3}{2}x_1^2x_2- \frac{1}{2} x_2^3$. 
	The polynomial $zv(z)-u(z)$ has degree 4.  It has 4 roots, 2 of them are real and the other 2 are complex.   We have $R=R_1$, $R'=R_1'$. Then
\[\|f\|_{\sigma~~}\approx 0.7071, \quad \|f\|_{\sigma,\R}=0.5.\]
\end{example}

\begin{example}\label{exm:2:d:qubit} 
Let $f=\frac{1}{\sqrt{5}}x_1^3-\frac{3}{2\sqrt{5}}x_1^2x_3-\frac{3}{\sqrt{5}}x_1 x_2^3+ \frac{1}{2\sqrt{5}}x_2^3$.
	The polynomial $zv(z)-u(z)$ has degree  5.  It has 5 roots, 3 of them are real and the other 2 are complex.   We have $R=R_1$, $R'=R_1'$. Then
\[	\|f\|_{\sigma~~} \approx  0.7071, \quad \|f\|_{\sigma,\R}  \approx 0.5000.\]
\end{example}   
\subsection{Five examples from \cite{AMM10} and their variations}\label{subsec:AMM10}
In \cite{AMM10} the authors give examples of $d$-symmetric qubits for $d=4,\ldots,12$, which they assume to have the lowest complex spectral norm.  Their examples are motivated by the Majorana model, see Appendix 3.  (Note that some examples have at least two versions (a) and (b).)  Our software gave the same values of the spectral norms for the examples in \cite{AMM10}.  We could not find with our software examples of symmetric  $d$-qubits with lower complex spectral norm.
In the following examples we find the spectral norm of $f_e$ of the form \eqref{ftform}, where $f$ is the polynomial given in \cite{AMM10}.  The polynomial $h$ corresponds to a monomial given by Dicke basis with the lowest spectral norm.  (See Appendix 1.)
\begin{example}\emph{\cite[Corresponds to example 6.1]{AMM10}} \label{exm:9:d:qubit} Let $f= \frac{1}{\sqrt{3}}x_1^4+\frac{\sqrt{8}}{\sqrt{3}}x_1x_2^3$.
	The polynomial $zv(z)-u(z)$ has degree $10=3^2+1$. It has 10 roots, 4 of them are real and the other 6 are complex.  We have $R= R_1$, $R'=R_1'$.  Then
$\|f\|_{\sigma~~} \approx 0.5774$.	According to \cite{AMM10}, $\|f\|_{\sigma}=\frac{1}{\sqrt{3}}$.  
Let $h= \sqrt{6}x_1^2x_2^2$.  Then $\|h\|_{\sigma}=\frac{\sqrt{3}}{\sqrt{8}}\approx 0.6124$.  (See \eqref{specnrmSj1jn}.)  Table 1 gives the results for $\|f_e\|_{\sigma}$:
 
\begin{table*}[htb]
	\centering
	\begin{scriptsize}
		\begin{tabular}{|c|c|c|c|c|c|c|c|c|}  \hline
			$\|f_e\|_{\sigma}$	& $t= \frac{1}{5}$	     & $t= \frac{1}{4}$	   & $t= \frac{1}{3}$	& $t= \frac{1}{2}$	 \\ \hline
			$\omega = 1$   &                                     0.6787 & 0.7012 & 0.7358 & 0.7918        \\ \hline  
			$\omega = -1$   &                                   0.6314   & 0.6442 &  0.6645& 0.6989      \\ \hline  
			$\omega = i$   &                                        0.6662    & 0.6863 & 0.7172 &   0.7676    \\ \hline 
			$\omega = \frac{1}{2}+\frac{\sqrt{3}}{2}i$   &   0.6314  &  0.6442 &0.6645 &  0.6989    \\ \hline 
			$\omega = -\frac{1}{2}+\frac{\sqrt{3}}{2}i$   & 0.6787 &0.7012 & 0.7358 & 0.7918    \\ \hline 
		\end{tabular}
	\end{scriptsize}\caption{Computational Results of $\|f_e\|_{\sigma}$ for Example \ref{exm:9:d:qubit} with different $t$ and $\omega$.} \label{different:ab:results:qubit:order4}
\end{table*}  
\end{example}

\begin{example}\emph{\cite[Corresponds to example 6.2(b), figure 5(b)]{AMM10}} \label{exm:10:d:qubit}  Let $f=\frac{1}{\sqrt{1+A^2}} x_1^5 +\frac{\sqrt{5}A}{\sqrt{(1+A^2)}}x_1x_2^4$, where $A \approx 1.53154$.
	The polynomial $zv(z)-u(z)$ has degree $17=4^2+1$.  It has 17 roots, 5 of them are real and the other 12 are complex.  Four roots do not satisfy  (\ref{antifixqub1}).  Furthermore $R'=R_1'$.  Then $\|f\|_{\sigma~~} =0.5467$.
	According to \cite{AMM10}, $\|f\|_{\sigma}=\frac{1}{\sqrt{1+A^2}}$.
	Let $h=\sqrt{10}x_1^3x_2^2$.  Then $\|h\|_{\sigma}=\frac{6\sqrt{6}}{25}\approx 0.5879$.  (See \eqref{specnrmSj1jn}.)  Here is the table for $\|f_e\|_{\sigma}$:
	\begin{table*}[htb]
		\centering
		\begin{scriptsize}
			\begin{tabular}{|c|c|c|c|c|c|c|c|c|}  \hline
				$\|f_e\|_{\sigma}$	& $t= \frac{1}{5}$	     & $t= \frac{1}{4}$	   & $t= \frac{1}{3}$	& $t= \frac{1}{2}$	 \\ \hline
				$\omega = 1$   &                                                0.5930&0.6038 &0.6214 &   0.6573      \\ \hline
				$\omega = -1$   &                                              0.5930 &0.6038 &0.6214&  0.6573       \\ \hline
				$\omega = i$   &                                                0.5622 &0.5692 &0.5793 &  0.5941         \\ \hline
				$\omega = \frac{1}{2}+\frac{\sqrt{3}}{2}i$   &   0.5759 &0.5830 & 0.5941 &  0.6133       \\ \hline
				$\omega = -\frac{1}{2}+\frac{\sqrt{3}}{2}i$   &    0.5759 &0.5830 &0.5941 &  0.6133     \\ \hline
			\end{tabular}
		\end{scriptsize}\caption{Computational Results of $\|f_e\|_{\sigma}$ for Example \ref{exm:10:d:qubit} with different $t$ and $\omega$.} \label{different:ab:results:qubit:order5}
	\end{table*}	

\end{example}

\begin{example}\emph{\cite[Corresponds to example 6.3]{AMM10}} \label{exm:11:d:qubit} Let $f=\sqrt{3}(x_1^5 x_2 + x_1 x_2^5)$.
	The polynomial $zv(z)-u(z)$ has degree $25<5^2+1$.  It has 25 roots, 7 of them are real and the other 18 are complex. 	For the 25 roots $z$, (\ref{antifixqub1}) fails to hold for 5 roots. Six of the seven real roots satisfy $zq(z)-p(z)=0$.  Then
$\|f\|_{\sigma~~} =0.4714$. 
	According to \cite{AMM10}, $\|f\|_{\sigma}=\frac{\sqrt{2}}{3}$.  	Let  $h=\sqrt{20}x_1^3x_2^3$.  Then $\|h\|_{\sigma}=\frac{\sqrt{5}}{4}\approx 0.5590$.  (See \eqref{specnrmSj1jn}.)  Here is the table for $\|f_e\|_{\sigma}$:
	
	\begin{table*}[htb]
		\centering
		\begin{scriptsize}
			\begin{tabular}{|c|c|c|c|c|c|c|c|c|}  \hline
				$\|f_e\|_{\sigma}$	& $t= \frac{1}{5}$	     & $t= \frac{1}{4}$	   & $t= \frac{1}{3}$	& $t= \frac{1}{2}$	 \\ \hline
				$\omega = 1$                                               & 0.5382 &0.5590 &0.5946 &0.6545 \\ \hline
				$\omega = -1$                                             & 0.5382 &0.5590 & 0.5946 &0.6545  \\ \hline
				$\omega = i$                                               &0.4777  &0.4811 &0.4886 &  0.5076  \\ \hline
				$\omega = \frac{1}{2}+\frac{\sqrt{3}}{2}i$    &  0.5054 &0.5148 &0.5312 & 0.5688\\ \hline
				$\omega = -\frac{1}{2}+\frac{\sqrt{3}}{2}i$   & 0.5054  &0.5148 & 0.5312 & 0.5688\\ \hline
			\end{tabular}
		\end{scriptsize}\caption{Computational Results of $\|f_e\|_{\sigma}$ for Example \ref{exm:11:d:qubit} with different $t$ and $\omega$.} \label{different:ab:results:qubit:order6}
	\end{table*}
 
\end{example}

\begin{example}\emph{\cite[Corresponds to example 6.4]{AMM10}} \label{exm:12:d:qubit} Let $f=\frac{\sqrt{7}}{\sqrt{2}}(x_1^6x_2+x_1x_2^6)$. 	
	The polynomial $zv(z)-u(z)$ has degree $36<6^2+1$.  It has 36 roots, 6 of them are real and the other 30 are complex.  For the 36 roots $z$, (\ref{antifixqub1}) fails to hold for 11 roots.  Five of the six real roots satisfy $zq(z)-p(z)=0$.  Then	
$\|f\|_{\sigma~~} =0.4508$.  Let $h=\sqrt{35}x_1^4x_2^3$.
Then $\|h\|_{\sigma}=\frac{48\sqrt{15}}{7^3}\approx 0.5420$.  (See \eqref{specnrmSj1jn}.)  
Here is the table for $\|f_e\|_{\sigma}$:
\begin{table*}[htb]
		\centering
		\begin{scriptsize}
			\begin{tabular}{|c|c|c|c|c|c|c|c|c|}  \hline
				$\|f_e\|_{\sigma}$	& $t= \frac{1}{5}$	     & $t= \frac{1}{4}$	   & $t= \frac{1}{3}$	& $t= \frac{1}{2}$	 \\ \hline
				$\omega = 1$                                               & 0.5006 & 0.5131 &0.5346 & 0.5796\\ \hline
				$\omega = -1$                                             & 0.4939 &0.5048 & 0.5229 & 0.5597 \\ \hline
				$\omega = i$                                               & 0.4988 &0.5109 & 0.5314 & 0.5742 \\ \hline
				$\omega = \frac{1}{2}+\frac{\sqrt{3}}{2}i$    & 0.4998 & 0.5121 &0.5332 &  0.5772 \\ \hline
				$\omega = -\frac{1}{2}+\frac{\sqrt{3}}{2}i$   &0.4975 &0.5092 &0.5291 & 0.5703 \\ \hline
			\end{tabular}
		\end{scriptsize}\caption{Computational Results of $\|f_e\|_{\sigma}$ for Example \ref{exm:12:d:qubit} with different $t$ and $\omega$.} \label{different:ab:results:qubit:order7}
	\end{table*}

\end{example}

\begin{example} \emph{\cite[Corresponds to example 6.5]{AMM10}}
	\label{exm:13:d:qubit}  Let $f=4(0.336\sqrt{2}x_1^7x_2+0.3705\sqrt{7}x_1^2x_2^6)$.
	The polynomial $zv(z)-u(z)$ has degree $42<7^2+1$, which has 41 roots, 7 of them are real and the other 34 are complex. One of the real roots $z=0$ has multiplicity 2.   For the 41 roots $z$, (\ref{antifixqub1}) fails to hold for 10 roots. Then
$\|f\|_{\sigma~~} =0.4288$. Let $h=\sqrt{70}x_1^4x_2^4$.
Then $\|h\|_{\sigma}=\frac{\sqrt{70}}{16}\approx 0.5229$.  (See \eqref{specnrmSj1jn}.)  
Here is the table for $\|f_e\|_{\sigma}$:
\begin{table*}[htb]
		\centering
		\begin{scriptsize}
			\begin{tabular}{|c|c|c|c|c|c|c|c|c|}  \hline
				$\|f_e\|_{\sigma}$	& $t= \frac{1}{5}$	     & $t= \frac{1}{4}$	   & $t= \frac{1}{3}$	& $t= \frac{1}{2}$	 \\ \hline
				$\omega = 1$                                               & 0.4946 & 0.5108&0.5374& 0.5867\\ \hline
				$\omega = -1$                                             & 0.4841 &0.4979 &0.5206 & 0.5630 \\ \hline
				$\omega = i$                                               & 0.4919 & 0.5075 & 0.5330& 0.5806 \\ \hline
				$\omega = \frac{1}{2}+\frac{\sqrt{3}}{2}i$    & 0.4934 &0.5093 &0.5354 & 0.5840\\ \hline
				$\omega = -\frac{1}{2}+\frac{\sqrt{3}}{2}i$   & 0.4898 &0.5049 & 0.5297 & 0.5759\\ \hline
			\end{tabular}
		\end{scriptsize}\caption{Computational Results of $\|f_e\|_{\sigma}$ for Example \ref{exm:13:d:qubit} with different $t$ and $\omega$.} \label{different:ab:results:qubit:order8}
	\end{table*}

\end{example}

More details on the above examples and additional examples for $n=2$ are given in \cite{FW16}.
\subsection{A family of symmetric $3$-qutrits}\label{exsym3qtr}  Let us consider the following family of symmetric $3$-qutrits.
\begin{eqnarray*}
f_{a,b}(x_1,x_2,x_3)=a(x_1^3+x_2^3+x_3^3)+bx_1x_2x_3, \quad 3|a|^2+|b|^2/6=1.
\end{eqnarray*}
This example is inspired by \cite[Example (1)]{CGL99}.   It is straightforward to show 
that 
$$\|f_{1/\sqrt{3},0}\|_{\sigma}=\|f_{1/\sqrt{3},0}\|_{\sigma,\R}=\frac{1}{\sqrt{3}}\approx 0.5774, \quad \|f_{0,\sqrt{6}}\|_{\sigma}=\|f_{0,\sqrt{6}}\|_{\sigma,\R}=\frac{\sqrt{2}}{3}=0.4714.$$
Note that $f_{0,\sqrt{6}}$ corresponds to the most entangled Dicke basis $\cS(3,3)$ \eqref{defSdn}.  Observe that $\|f_{a,b}\|_{\sigma}=\|f_{\bar a,\bar b}\|_{\sigma}=\|f_{\zeta a,\zeta b}\|_{\sigma}$ for $|\zeta|=1$.  Here is the table for $\|f_{a,b}\|_{\sigma}$:
  
\begin{table*}[htb]
	\centering
	\begin{scriptsize}
		\begin{tabular}{|c|c|c|c|c|c|c|c|c|}  \hline
			$a$	&	$b$ &  No. of real fixed points & No. of complex fixed points & $\|f_{a,b}\|_{\sigma,\R}$ & $\|f_{a,b}\|_{\sigma}$ \\ \hline
			$\frac{1}{3}$ &  2 & 8& 56& 0.5774 & 0.5774 \\ \hline   
			$\frac{1}{2}$ & $\sqrt{\frac{3}{2}}$& 8& 56&0.5244&0.5244\\ \hline 
			$\frac{1}{3}$ &  -2 & 8& 56 &0.4975& 0.5092\\ \hline   
			$\frac{1}{2}$ & -$\sqrt{\frac{3}{2}}$& 8& 56 &0.5000 &0.5000\\ \hline 
			$0$ &  $\sqrt{6}$ & 5& 50& 0.4714 & 0.4714 \\ \hline   
			$\frac{1}{\sqrt{3}}$ & 0& 8& 56   &0.5774& 0.5774\\ \hline 
				$\frac{1}{6}+\frac{\sqrt{3}}{6}i$& $\sqrt{2}-\sqrt{2}i$ & 1 &63 &-&0.5730  \\ \hline 
	$\frac{1}{4}+\frac{\sqrt{3}}{4}i$& $\frac{\sqrt{6}}{4}+\frac{3\sqrt{2}}{4}i$ & 1 &63 & -& 0.5244\\ \hline		
		\end{tabular}
	\end{scriptsize}\caption{Computational Results for  \ref{exsym3qtr} with different $a$ and $b$.} \label{tensor:computation:results:exm:ab:real}
\end{table*}
 
Let us consider the real case. For different real $a$ and $b$, we use Bertini to solve the equation $F(y) = y$, and get the real spectral norm of tensor $f_{a,b}$. Results are shown in the Table \ref{tensor:computation:results:exm:ab:real}.  As we see, the number of real fixed point is $8=(3-1)^3$ the maximum possible as given by part (3) of Theorem \ref{theofoundthm}, except for $f_{0,\sqrt{6}}$ corresponding to $\cS(3,3)$.
Thus, among all these examples the Dicke state $\cS(3,3)$ is the most entangled one.
  
\subsection{A family of symmetric $4$-ququadrits}\label{exsym4qtr}
Recall that most entangled $3$-qubit is $\cW\in\rS^3\C^2$, which corresponds to $f=\sqrt{3}x_1^2x_2$ \cite{TWP09,CXZ10}.  (It is the Dicke basis element  $\cS(3,2)$.) 
It is of interest to consider the tensor product state $\cW\otimes \cW\in \otimes^6\C^2$ \cite{CCDW,CF18}.  Note that this state is not symmetric.  Lemma 2.3 in \cite{DFLW17} yields
$$\|\cW\otimes\cW\|_{\sigma}=\|\cW\|_{\sigma}^2=\|\cS(3,2)\|^2=(\frac{2}{3})^2=\frac{4}{9}\approx 0.4444.$$
It is possible to represent $\cW\otimes\cW$ as a symmetric tensor $\cS\in \rS^3\C^4$.
It is represented by polynomial $f=x_1^2 x_4+2x_1x_2x_3$ \cite{CCDW}.
Let us consider $f_{a,b}=ax_1^2 x_4+2bx_1x_2x_3$,	
where $|a|^2+2|b|^2 = 3$.  (So $\|f_{a,b}\|=1$.)  Note that $f=f_{1,1}$.   In Table \ref{tensor:order3:n4:exm:results:ab}, we show $\|f_{a,b}\|_{\sigma}$ for some values of $a,b$.
\begin{table*}[htb]
	\centering
	\begin{scriptsize}
		\begin{tabular}{|c|c|c|c|c|c|c|c|c|}  \hline
			$a$	&	$b$ &  $\|f_{a,b}\|_{\sigma}$ \\ \hline
			1 & 	1   & 0.4444  \\ \hline 
			$\frac{\sqrt{2}}{2}$&  	$\frac{\sqrt{5}}{2}$ &  0.4536\\ \hline
			$\frac{\sqrt{5}}{2}$ &	$\sqrt{\frac{7}{8}}$&  0.4491  	\\ \hline 
					$\frac{\sqrt{2}}{2}+\frac{\sqrt{2}}{2}i$ & 	$\frac{1}{2}+\frac{\sqrt{3}}{2}i$   &   0.4444 \\ \hline   
			$\frac{\sqrt{2}}{2}(\frac{1}{2}-\frac{\sqrt{3}}{2}i)$ &$\frac{\sqrt{5}}{2}(\frac{\sqrt{3}}{4}+\frac{\sqrt{13}}{4}i)$ & 0.4536\\ \hline
		$\frac{\sqrt{6}}{2}(\frac{\sqrt{2}}{2}+\frac{\sqrt{2}}{2}i)$	 & $\frac{\sqrt{3}}{2}(\frac{\sqrt{3}}{2}+\frac{1}{2}i)$	 & 0.4714\\ \hline
		\end{tabular}
	\end{scriptsize}\caption{Computational Results for Example \ref{exsym4qtr} with different $a$ and $b$.} \label{tensor:order3:n4:exm:results:ab}
\end{table*} 
These computations point out that probably $\|\cS\|=\|f\|_{\sigma}$ is equal to $\|\cW\otimes\cW\|_{\sigma}$. 	 This contrast with the results in \cite{CCDW,CF18} that $\mathrm{rank }\; \cS=7<\mathrm{rank }\;\cW\otimes\cW=8 <(\mathrm{rank }\;\cW)^2=9$.   (Recall that the rank of a tensor $\cT$ is the minimal number of summands in a representation of  
$\cT$ as a sum of rank one tensors.)  Next observe that $f$ does not have the minimal complex spectral norm in the above examples.  Finally, $\|f\|_{\sigma}<\|\cS(3,4)\|=\frac{\sqrt{2}}{3}$.\\

\emph{Acknowledgement}:  We thank the two referees for their useful remarks and comments.

$ $
\setcounter{theorem}{0}
\appendix{Appendix 1: Dicke states and their entanglement}\label{app: dicke}\\

Let $\cS(\bj), \bj\in J(d,n)$ be the orthonormal basis in $\rS^d\F^n$ given
in the end of \S\ref{sec:symten}.  
For $n=2$ this basis is called  the Dicke basis \cite{Dic}.  Recall that $\cS(\bj)\times\otimes^d\x=\sqrt{c(\bj)} \x^{\bj}$.

Let
\begin{eqnarray}\label{defSdn}
		&&\cS(d,n)=\cS(\bj), \; \bj=(j_1,\ldots,j_n)\in J(d,n), \textrm{ where}\\ 
		&&j_1=\cdots =j_l=\lfloor\frac{d}{n}\rfloor,\;j_{l+1}=\cdots = j_{n}=\lceil\frac{d}{n}\rceil, \;l=n\lceil\frac{d}{n}\rceil - d.\notag
		\end{eqnarray}
In the following lemma we find the entanglement of each $\cS(\bj), \bj\in J(d,n)$ and the maximum entanglement of these states.
\begin{lemma}\label{lem:Sdn1} Assume that $n,d\ge 2$ are two positive integers.  Then
	\begin{enumerate}
	\item  For each $\bj=(j_1,\ldots,j_n)\in J(d,n)$ the following equality holds
		\begin{eqnarray*}\label{entanfSj1jn}
		\eta(\cS(\bj))=\log_2d^d - \log_2 d! + \sum_{k=1}^n (\log_2 j_k!-\log_2 j_k^{j_k}).
		\end{eqnarray*}
\item
\begin{eqnarray*}\label{specnrmSdn}
 \|\cS(d,n)\|_{\sigma}=\sqrt{ \frac{ d!\big(\lfloor\frac{d}{n}\rfloor\big)^{l\lfloor\frac{d}{n}\rfloor}\big(\lceil\frac{d}{n}\rceil\big)^{(n-l)\lceil\frac{d}{n}\rceil}}{d^d \big(\lfloor\frac{d}{n}\rfloor !\big)^l\big(\lceil\frac{d}{n}\rceil !\big)^{n-l}}}, \quad
 \eta(\cS(d,n))=\log_2 \frac{d^d \big(\lfloor\frac{d}{n}\rfloor !\big)^l\big(\lceil\frac{d}{n}\rceil !\big)^{n-l}}{ d!\big(\lfloor\frac{d}{n}\rfloor\big)^{l\lfloor\frac{d}{n}\rfloor}\big(\lceil\frac{d}{n}\rceil\big)^{(n-l)\lceil\frac{d}{n}\rceil}}.
 \end{eqnarray*}
		\item 
		\begin{eqnarray*}\label{specnrmSdnineq}
	\|\cS(d,n)\|_{\sigma}\le\|\cS(\bj)\|_{\sigma},\quad\eta(\cS(\bj))\le \eta(\cS(d,n)), \quad \textrm{for each }\bj\in J(d,n).
		\end{eqnarray*}
		\item Assume that the integer $n\ge 2$ is fixed and $d\gg 1$.  Then
		\begin{eqnarray*}\label{asforetaSdn}
		\eta\big(\cS(d,n)\big)=\frac{1}{2}\big((n-1)\log_2 d - n\log_2 n\big)+O(\frac{1}{d}).
		\end{eqnarray*}
	\end{enumerate}
\end{lemma}

\begin{proof} (1)
	Clearly
	\[|\cS(j_1,\ldots,j_n)\times\otimes^d\x|^2=\frac{d!}{j_1!\cdots j_n!}(|x_1|^2)^{j_1}\cdots (|x_n|^2)^{j_n}.\]
	Use Lagrange multipliers to deduce that the maximum of the above function for $\|\x\|=1$ is achieved at the points $|x_k|^2=\frac{j_k}{j_1+\cdots+j_n}=\frac{j_k}{d}$ for $k\in[n]$.  Banach's theorem \eqref{Banchar} yields
	\begin{eqnarray}\label{specnrmSj1jn}
	\|\cS(j_1,\ldots,j_n)\|_{\sigma,\R}^2=\|\cS(j_1,\ldots,j_n)\|_{\sigma}^2=\frac{d!\prod_{k=1}^n j_k^{j_k}}{d^d\prod_{k=1}^n j_k!}.
	\end{eqnarray}
	This establishes the expression for $\eta(\cS(\bj))$.
	
	\noindent
	(2) Follows straightforward from the definition of $\cS(d,n)$ in \eqref{defSdn} and the proof of part (1).
	
	\noindent
	(3)  Let $a,b$ be nonnegative integers such that $a\le b-2$.  We claim that 
	\[\frac{a!b!}{a^a b^b}< \frac{(a+1)!(b-1)!}{(a+1)^{a+1}(b-1)^{b-1}}.\]
	Indeed, the above inequality is equivalent to 
	\[\frac{a!(a+1)^{a+1}}{a^a(a+1)!}<\frac{(b-1)!b^b}{(b-1)^{b-1}b!}\iff \left(\frac{a+1}{a}\right)^a <\left(\frac{b}{b-1}\right)^{b-1}.\]
	As $0^0=1$ we deduce that the above inequalities hold for $a=0$ and $b\ge 2$.  Assume that $a\ge 1$.
	Then the last inequality in the above displayed relation is equivalent to the well known statement that the sequence $(1+\frac{1}{m})^m$ is a strictly increasing .  
	
	Consider $\|\cS(j_1,\ldots,j_n)\|^{-2}$.  Suppose that there exists $j_p,j_q$ such that $|j_p-j_q|\ge 2$.  Without loss of generality we may assume that $j_p\le j_q-2$.
	Let $j'_l=j_l$ for $l\in [n]\setminus\{p,q\}$, and $j'_p=j_p+1, j'_q=j_q-1$.  Then the above inequality yields that $\|\cS(j_1,\ldots,j_n)\|^{-2}<\|\cS(j_1',\ldots,j_n')\|^{-2}$.
	Hence the maximum value of $\|\cS(j_1,\ldots,j_n)\|^{-2}$, where $\{j_1,\ldots,j_n\}\in J(d,n)$, is achieved for $\{j_1,\ldots,j_n\}$ satisfying $|j_p-j_q|\le 1$ for all $p,q\in[n]$.  Without loss of generality we can assume that $j_1,\ldots,j_n$ are given as in \eqref{defSdn}.  This shows the inequality $\|\cS(d,n)\|_{\sigma}\le\|\cS(\bj)\|_{\sigma}$.   The definition \eqref{defetaT} of $\eta(\cS)$ yields the inequality $\eta(\cS(\bj))\le \eta(\cS(d,n))$ for each $\bj\in J(d,n)$.
	
	\noindent
	(4).  The expression for $\eta(\cS(d,n))$ follows from Sterling's formula \cite[p. 52]{Fel58}
	$k!=\sqrt{2\pi k} k^k e^{-k}e^{\theta_k/12 k}$, where $0<\theta_k<1$. 	
\end{proof}

We now comment on the results given in Lemma \ref{lem:Sdn1}.  Parts (1)-(3)  are well known for $n=2$ in physics community \cite{AMM10}.  The states $\cS(j_1,j_2)$ are called \emph{Dicke} states.
Note that
\[
\|\cS(3,2)\|_{\sigma}=\frac{2}{3}\approx 0.6667,\;\|\cS(4,2)\|_{\sigma}=\frac{\sqrt{6}}{4}\approx 0.61237,\;\|\cS(5,2)\|_{\sigma}=\frac{6\sqrt{6}}{25}\approx 0.5879.\]
It is known that the most entangled $3$-qubit state with respect to geometric measure is $\cS(3,2)$ \cite{TWP09,CXZ10}.  That is, the spectral norm of a nonsymmetric $3$-qubit
is not less than the spectral norm of $\cS(3,2)$, which is equivalent to the equality
$\eta((2,2,2))=\eta(\cS(3,2))$ as in (\ref{maxentn}). However, for $d>3$ the examples that are given in \S\ref{sec:numerex} show  that the states $\cS(d,2)$ 
 are not the most entangled states in $\rS^d\C^2$.   See the Examples $i$ with , denoted as $\cE_i$, in \S\ref{sec:numerex} for $i=4-8$. 
 
 It is shown in  \cite{DFLW17} that the most entangled $4$-qubit is $\cM_4\in\otimes^4\C^2$, which is not symmetric.  (This was a conjecture of 
 Higuchi-Sudbery \cite{HS00}.)   Note that 
 \begin{eqnarray*} \|\cM_4\|_{\sigma}=\frac{\sqrt{2}}{3}\approx 0.4714 <\|\cE_5\|_{\sigma}.
 \end{eqnarray*} 
According to \cite{AMM10}, the symmetric state $\cE_5$ is the most entangles $4$-symmetric qubit.

Lemma 4.3.1 in \cite{Ren05} yields that
\begin{equation}\label{upbndentsym}
\eta(\cS)\le \log_2 {n+d-1\choose n-1}, \quad \cS\in\rS^d\C^n, \|\cS\|=1.
\end{equation}
(See also \cite{MGBB10}.)   In particular, for $n=2$ we have the inequality:
\begin{equation}\label{qubitsentangi}
\eta(\cS)\le \log_2(d+1), \quad \cS\in\rS^d\C^2, \|\cS\|=1.
\end{equation}

Note that for a fixed $n$ and large $d$ we have the complexity expression
\begin{eqnarray*}\label{asymptfornchos}
\log_2 {n+d-1\choose n-1}= (n-1) \log_2 (d+1) +\frac{(n-1)(n-2)}{2\ln 2} - \log_2 (n-1)!+O\left(\frac{1}{d}\right).
\end{eqnarray*}

Let 
\begin{eqnarray*}\label{defetadn}
\eta_{sym}(d,n)=\max\{\eta(\cS),\; \cS\in\rS^d\C^n, \|\cS\|=1\}.
\end{eqnarray*}
Combining the inequality (\ref{upbndentsym})  with part (3) of Lemma \ref{lem:Sdn1}  we obtain
\begin{equation}\label{lowupbdsetadn}
\log_2 \frac{d^d \big(\lfloor\frac{d}{n}\rfloor !\big)^l\big(\lceil\frac{d}{n}\rceil !\big)^{n-l}}{ d!\big(\lfloor\frac{d}{n}\rfloor\big)^{l\lfloor\frac{d}{n}\rfloor}\big(\lceil\frac{d}{n}\rceil\big)^{(n-l)\lceil\frac{d}{n}\rceil}}\le \eta_{sym}(d,n)\le  \log_2 {n+d-1\choose n-1}, \;l=n\lceil\frac{d}{n}\rceil -d.
\end{equation}
In particular
\begin{equation}\label{lowupbdsetad2}
\log_2\frac{d^d\big(\lfloor\frac{d}{2}\rfloor!\big)\big(\lceil\frac{d}{2}\rceil!\big)}{d! \big(\lfloor\frac{d}{2}\rfloor\big)^{\lfloor\frac{d}{2}\rfloor} \big(\lceil\frac{d}{2}\rceil\big)^{\lceil\frac{d}{2}\rceil}}\le \eta_{sym}(d,2)\le \log_2(d+1).
\end{equation}

There is a gap of factor $2$ between the lower and  the upper bounds in  (\ref{lowupbdsetadn}) and (\ref{lowupbdsetad2}) for fixed $n$ and $d\gg 1$.
In \cite{FK16} it is shown that the following inequality holds with respect to the corresponding Haar measure on the unit ball $\|\cS\|=1$ in $\rS^d\C^2$:
\begin{eqnarray*}\label{coplexentangsym1}
\Pr(\eta(\cS)\le \log_2 d -\log_2 (\log _2 d) +\log_2 4 -\log_2 5)\le \frac{1}{d^6}.
\end{eqnarray*}
This shows that the upper bounds in (\ref{lowupbdsetad2}) have the correct order.  In particular, the above inequality is the analog of the inequality $\eta(\cT)\ge d - 2\log_2 d-2$ for most $d$-qubit states in \cite{GFE09}.\\ 

\appendix{Appendix 2: Computing isolated roots of certain polynomial systems}\label{sec:isoroots}\\

Let $g=g(\x)$ be a nonconstant polynomial in  $\x=(x_1,\ldots,x_m)\trans\in\F^m$.  Denote by $d=\deg g$ the total degree of $g$.  Then $g_{\pi}\in \rP(d,m,\F)\setminus\{0\}$ denotes the homogeneous part of $d$ of degree $d$.  That is, $g-g_{\pi}$ is a polynomial of degree less than $d$.  In this Appendix we assume that $g_i(\x)$ is a nonconstant polynomial of total degree $d_i=\deg g_{i,\pi}\in\N$ for $i\in[m]$.   
Our next assumption is that the following system of homogeneous equations has only the trivial solution
\begin{eqnarray}\label{trivsolcond}
g_{1,\pi}(\x)=\cdots=g_{m,\pi}(\x)=0, \;\x\in \C^m \Rightarrow \x=\0.
\end{eqnarray} 
It is well known that the above assumption yields that the system
\begin{eqnarray}\label{regsys}
g_1(\x)=\cdots =g_m(\x)=0, \;\x\in\C^m
\end{eqnarray}
has exactly $D=\prod_{i=1}^m d_i$ solutions, counted with multiplicities, and no solutions at infinity \cite{Fri77}.  

We now give a simple condition which imply that the system \eqref{regsys} has $D$ distinct solutions.  
%
\begin{lemma}\label{simplerootcond}  Let $g_i(\x), \x=(x_1,\ldots,x_m)\trans\in\C^m$ be a polynomial of degree $d_i\ge 1$ for $i\in[m]$.  Assume that  $D=\prod_{i=1}^m d_i>1$ and the system \eqref{regsys}
satisfies the condition \eqref{trivsolcond}.  
Then the system \eqref{regsys} has $D$ distinct solutions if and only if the system 
\begin{eqnarray}\label{simplerootcond1}
g_i(\x)=0,\; i\in[m],\quad
\det [\frac{\partial g_i} {\partial x_j}(\x)]_{i,j\in[m]}=0 
\end{eqnarray}
is not solvable over $\C^n$.
\end{lemma}
\begin{proof}  Let $\bG=(g_1,\ldots,g_m): \C^m \to \C^m$.  Recall that the assumptions that  the system \eqref{regsys} satisfies the condition \eqref{trivsolcond}
implies that $G$ is a proper map.   Assume that $\bG(\zeta)=\0$.   Then $\zeta$ is a simple solution of $G(\x)=\0$ if and only if $\det [\frac{\partial g_i} {\partial x_j}(\x)]_{i,j\in[m]}\ne 0$.  Hence, if the system \eqref{regsys} has $D$ distinct solutions it follows that the above system is not solvable.

Vice versa, assume that the above system is not solvable.
Then the Jacobian matrix $\rD(\bG)$ is invertible at $\zeta$.  Therefore $\bG$ is a local invertible map in the neighborhood of $\zeta$.  In particular $\zeta$ is a simple zero of $\bG$.  As $\bG^{-1}(\0)$ has $D$ points counting with their multipliciites, it follows that $|\bG^{-1}(\0)|=D$.
\end{proof}
\begin{definition}\label{defsimpsys}  A system of polynomial equations \eqref{regsys} is called simple if the following conditions hold:
\begin{enumerate}
\item For each $i\in[m]$ the inequality $d_i=\deg g_i\ge 1$ holds.
\item The total degree of the system $D=\prod_{i=1}^m d_i$ is greater than $1$.
\item The condition \eqref{trivsolcond} holds.
\item The system \eqref{regsys} has $D$ distinct solutions.
\end{enumerate}
The system  \eqref{regsys} is called $x_1$-simple if in addition to the above conditions
the $x_1$-coordinates of $D$ solutions are distinct.
\end{definition}
An example of a simple system system  is
\begin{eqnarray*}
\frac{1}{d_i}x_i^{d_i}-x_i=0, \quad 1<d_i\in\N, \;i\in[m]. 
\end{eqnarray*}
An example of $x_1$-simple system is obtained from the above system by replacing the $g_1(\x)=\frac{1}{d_1}x_1^{d_i}-x_1$ with $g_1(\x)=\frac{1}{d_1}(x_1+\sum_{i=2}^m a_ix_i)^{d_1} -x_1$.  It is straightforward to show that this system is $x_1$-simple if and only if $(a_2,\ldots,a_m)\trans$ is not in a corresponding variety $V\subset \C^{m-1}$. Indeed,
find all possible values of $x_2,\ldots,x_m$ to obtain the monic polynomial $p_1(x_1,a_2,\ldots,a_m)$ of degree $D$ that its zeros are the $x_1$-coordinates of the corresponding solution of this system.  Then $p_1(x_1,a_2,\ldots,a_m)$ has $D$ simple solutions if and only if the discriminant of $p_1$ is not zero.

We now discuss briefly available methods to find all solutions of an $x_1$-simple system \eqref{regsys} and their  arithmetic and bit complexities.  Recall that arithmetic complexity counts the number of operations to execute a given algorithm, where each operation is counted as one unit of time.  The bit complexity takes also in account the number of bits required for each operation, as addition/subtraction and product/division, where the length of the bits are count.  Let $\cO$ and $\cO_{B}$ denote the arithmetic and the  bit complexities respectively.  By $\tilde \cO$ and $\tilde \cO_{B}$ we denote the arithmetic and the bit complexities, where some logarithm terms are dropped.

One of the standard method to solve system \eqref{regsys} is to use the reduced  Gr\"obner basis with respect to an ordering of monomials that is preserved under product.  See for example \cite{BFS15}. The simplest order is a lexicographical order $x_1\prec x_2\prec\cdots\prec x_m$.  We now recall a simple case of the Shape Lemma that we are using in this paper \cite{Rou99}:
\begin{lemma}\label{shapelem}(Shape Lemma)  Let \eqref{regsys} be a simple system. Then the last polynomial in the reduced  Gr\"obner basis is a monic polynomial  $p_1(x_1)$ of degree $D$.  If  the system \eqref{regsys} is $x_1$-simple  then the reduced Gr\"obner basis is of the form: $p_1(x_1), x_2-p_2(x_1), \ldots,x_m-p_m(x_1)$, and $\deg p_i< D$ for $i=2,\ldots,m$.

If in addition to the above assumptions $g_1,\ldots,g_m$ have real coefficients then $p_1,\ldots,p_m$ have real coefficients.
\end{lemma}

In the rest of this Appendix what we assume that the system \eqref{regsys} is $x_1$-simple.  Then all solutions of \eqref{regsys} are of the form $(x_1,p_2(x_1),\ldots,p_n(x_1))$, for all zeros $x_1$ of $p_1$.  Assume furthermore that $g_1,\ldots,g_n$ have real coefficients.  Then all real solutions of an $x_1$-simple systems are of the form  $(x_1,p_2(x_1),\ldots,p_m(x_1))$, for all real zeros $x_1$ of $p_1$. 

In what follows we consider an $x_1$-simple system of polynomial equations \eqref{regsys}, where 
\begin{eqnarray}\label{d1=dn=d}
d_1=\cdots =d_m=d-1\ge 2, 
\end{eqnarray}
with  coefficients either in the domain of Gaussian integers $\Z[\bi]$, or subdomain of integers $\Z$.  Let $\tau-1$ be the maximal number of bit size of the coefficients of the monomials in $g_1,\ldots,g_m$.  (Thus $\tau\in\N$.)
For our purposes we assume that $m$ is fixed and the positive integer $d$ varies.
First consider $p_1(x_1)$.   What is the arithmetic complexity to compute $p_1$?
 Proposition 1 in \cite{BFS15} gives $\cO(md^{3m-2})$.   (Assuming that the complexity of multiplication of two $p\times p$ matrices is $\cO(p^3)$.)  The bit complexity of finding $p_1(x_1)$ is $\tau$ times the arithmetic complexity times the size of the linear system to solve.  The linear system is of order $m(d-1)^m$.  Hence the bit complexity
 is $\cO_{B}(\tau m^2 d^{4m-2})$.  We next discuss the maximum bit size of the  coefficients of $p_i(x_1)$ for $i\in[m]$.  The arguments of \cite{MST17} yield that the maximum bit size is $\cO(m(d-1)^m(\tau +  \log d))$.  
 
 We now recall the complexity results for computations of the roots of $p_1(x_1)$ with  precision $2^{-\ell}$ for $\ell\in\N$.  Consider a polynomial equation over $\C$ of degree $N\ge 2$:
 \begin{eqnarray*}\label{onepoleq}
f(x)=x^N +\sum_{i=1}^{N} a_i x^{N-i}=0, \quad a_i\in \C, i\in[N].
\end{eqnarray*}

Assume that $|a_i|< 2^{\beta}$ for $i\in[N]$ and $\beta\in\N$.
Then the arithmetic complexity of computing all roots of the above equation with precision  $2^{-\ell}$ is $\cO(N\log^5 N \log (\max(N,\beta)+\ell))$.  The bit complexity is equal to the arithmetic complexity times $M(2N^2\max(N,\beta)+\ell)$ \cite{NR96}.  Here
\begin{eqnarray*}\label{deffuncM}
M(t)=\cO(t(\log t)\log\log t)
\end{eqnarray*}
is the complexity of multiplying two $t$ bit integers.  That is, the bit complexity of computing all roots of $f$ with precision $2^{-\ell}$ is at most of order
\begin{eqnarray*}\label{bticomplunpol} 
M(2N^2\max(N,\beta)+\ell)N\log^5 N \log (\max(N,\beta)+\ell).
\end{eqnarray*}

We now give the complexity of finding each coordinate of each solution of the system \eqref{regsys}, satisfying the condition  \eqref{d1=dn=d},  with a relative precision $\varepsilon$:
\begin{lemma}\label{eransol}  Consider an $x_1$-simple system of polynomial equations with coefficients in $\Z[\bi]$, which satisfies the condition \eqref{d1=dn=d}.  Let  $\tau-1$ be the maximal number of bit size of the coefficients of the monomials in $g_1,\ldots,g_m$.   Then the bit complexity of computing the value of each coordinate of each solution of the system with precision $2^{-\ell}, \ell\in\N$ is of at most of order 
\begin{eqnarray*}
\tau m^2 d^{4m-2}+M(2N^2\max(N,\beta)+\ell')N\log^5 N \log (\beta+\ell'),
\end{eqnarray*}
where
\begin{eqnarray*}
N=(d-1)^m, \;\beta=\cO(m(d-1)^m(\tau +  \log d)),\;\ell'=\ell+N(\beta+1).
\end{eqnarray*}
\end{lemma}
\begin{proof}
We first compute $p_1(x_1)$.  The bit complexity of computing $p_1(x_1)$  $\cO(\tau m^2 d^{4m-2})$.  Next we compute all roots of $p_1(x_1)$ with precision $2^{-\ell'}$ for $\ell'\ge \ell$ to be determined later.  Recall that the degree of the monic polynomial $p_1(x_1)$ is $N=(d-1)^m$, and  the maximum bits of its coefficients is bounded by $\beta-1$, where $\beta=\cO(m(d-1)^m(\tau +  \log d))$.  Hence the bit complexity  of finding the roots of $p_1(x_1)$ is $\cO( M(2N^2\max(N,\beta)+\ell')N\log^5 N \log (\beta+\ell'))$.  Observe next that any root $z$ of $p_1(x_1)$ satisfies the inequality $|z|^N\le 2^{\beta-1}\frac{|z|^N-1}{|z|-1}$.  Hence $|z|\le 2^{\beta}$.  It is left to estimate the error in the $k(>1)$  coordinate of a solution $\zeta=(\zeta_1,\ldots,\zeta_m)$ of the system.  Recall that $\zeta_k=p_k(\zeta_1)$, $\deg p_k<N$, and the bit size of each coefficient of  $p_k$ is at most $\beta-1$.
Let $ \eta_1$ an approximation of $\zeta_1$ that satisfies $|\zeta_1-\eta_1|\le 2^{-\ell'}\le 1/2$.  Then $|\eta_1|\le |\zeta_1|+1/2\le 2^{\beta}+1/2$.
 It is straightforward to show that
\begin{eqnarray*}
|p_k(\zeta_1)-p_k(\eta_1)|\le2^{\beta-1}|\zeta_1-\eta_1| \sum_{j=1}^{N-1}j\max(|\zeta_1|,|\eta_1|)^{j-1}\le |\zeta_1-\eta_1| 2^{N(\beta+1)}< 2^{N(\beta +1)-\ell'}.
\end{eqnarray*}
Hence for $\ell'=\ell +N(\beta +1)$ we obtain that $\eta_k=p_k(\eta_1)$ is an approximation of $\zeta_k$ with $2^{-\ell}$ precision. 
\end{proof}

\appendix{Appendix 3: The Majorana representation}\label{sec:app}\\

The geometrical measure of entanglement for qubits is studied extensively in physics literature \cite{AMM10,GFE09,HS00,Jungetall08,MGBB10,TWP09}.  In some papers, as \cite{AMM10,MGBB10}, the authors use the Majorana representation of symmetric qubits.  The paper \cite{AMM10} provides examples of what the authors believe the most entangled symmetric qubits for $d=4,\ldots,12$ based on Majorana representation.  As most mathematicians are not familiar with Majorana representation, we descibe in a few sentences the mathematical concepts of Majorama representation for the interested reader.  
We also describe briefly the main ideas behind the examples in \cite{AMM10}. 

Recall that a qubit state is identified with the class of all vectors ($d=1$):
\[\{\zeta\x: \x=(x_1,x_2)\trans\in\C^2, \|\x\|=1, \zeta\in\C,|\zeta|=1\}.\] 
Thus the set of all quibits  $\Gamma$ can be identified with the complex projective space $\bbP\C^2$.   $\Gamma$ can be also identified with the Riemann sphere $\C\cup\{\infty\}$.   Indeed, associate with a qubit $\x=(x_1,x_2)\trans, x_1\ne 0$ a unique complex number $z=\frac{x_2}{x_1}\in\C$.  The qubit $\x=(0,x_2), |x_2|=1$ corresponds to $z=\infty$.  
Recall that the set of $d$-symmetric quibit can be identified with the projective set $\bbP(\rS^d\C^2)$.  Equivalently,  the set of $d$-symmetric quibit can be identified with $\bbP(\rP(d,2,\C))$.  Take $f(\x)\in\rP(d,2,\C)\setminus\{0\}$.  Then $f(x_1,x_2)$ is a product of $d$ linear forms $f(x_1,x_2)=\prod_{i=1}^d(\zeta_{1,i}x_2-\zeta_{i,2}x_1)$.
Thus, there exists a unique correspondence between $[f]\in\bbP(\rP(d,2,\C))$ and the $d$ points $[(\zeta_{1,i},\zeta_{i,2})\trans]\in \bbP\C^2, i\in[d]$.  
This is the Majorana representation  \cite[\S4.4]{BZ06}.  

Next one replaces the Riemann sphere by the ordinary sphere $\rS^2\subset\R^3$.
(For example perform a stereographic projection of the unit sphere onto the complex plane.)  The geometric intuition suggests that the most entangled $d$-symmetric quibits correspond to the most evenly distributed $d$ points on $\rS^2$.
More precisely there are two optimal models:

\noindent
\textbf{T\'oth’s problem}: How $d$ points
have to be distributed on the unit sphere so that the minimum distance of all pairs of points becomes maximal \cite{Why52}.

\noindent
\textbf{Thomson’s problem}: How $d$ point charges can
be distributed on the surface of a sphere so that the potential energy is minimized \cite{Tho04}.

However, in certain cases as shown in \cite{AMM10} the most entagled symmetric states are do not solve neither of the above problems.

 \end{document}